\documentclass[11pt]{amsart}
\usepackage{amssymb,latexsym}
\usepackage{amsmath}
\usepackage{amsfonts}
\usepackage{amsthm}
\usepackage{cite}
\usepackage{enumerate}
\usepackage{graphicx}
\theoremstyle{plain}
\newtheorem{thm}{Theorem}
\newtheorem{corollary}{Corollary}

\newtheorem{lemma}{Lemma}
\newtheorem{proposition}{Proposition}
\theoremstyle{definition}
\newtheorem{definition}{Definition}
\theoremstyle{remark}
\newtheorem{remark}{Remark}

\theoremstyle{remark}
\newtheorem*{notation}{Notation}

\numberwithin{equation}{section}
\begin{document}
\title[Nonlinear Stochastic homogenization in Orlicz Spaces]
{Nonlinear Stochastic homogenization in Orlicz Spaces and applications}
\author{Dimitris Kontogiannis}
\address{Department of Mathematics\\
Iowa State University\\
Ames, IA 50010}
\email{dkontog@iastate.edu}
%\urladdr{http://www.public.iastate.edu/~dkontog/}
%\thanks{Research supported by the NSF under grant number

\date{}
%\begin{abstract}
%\end{abstract}
\maketitle
\section{Introduction}
%\label{S:intro}
In this work, we are interested in the stochastic homogenization of integral functionals defined in Orlicz-Sobolev spaces. We present a general version of nonlinear stochastic homogenization in these spaces and apply the general versions to homogenization problems in random media. In particular, we introduce a notion of $\Gamma-$convergence in Orlicz-Sobolev spaces and we show the compactness of a class $\mathcal{F}$ of integral functionals with respect to this convergence. To guarantee that the $\Gamma-$limit is a measure, one has to give a criterion called fundamental estimate which the class of functionals must satisfy. For classical $L^{p}-$spaces, such estimate was introduced by DeGiorgi and further developed by Dal Maso, Modica, Braides. A distance function is also defined as a metric so that the family of minimizers of these functionals is continuous with respect to this metric. Using results of ergodic theory, we prove a stochastic theorem concerning the limit of minimizers, which is an extension of \cite{DalMod86}. Finally, we apply the theorem to homogenization over a class of partial differential equations defined in Orlicz spaces. A technique of  such homogenization problems has been developed in \cite{Khrus05}. We improve these methods using continuum percolation models.  We also mention the possible improvement of existing results on homogenization of $p-$Laplace type equations\cite{MR2569903}. 

The paper is organized as follows.
\begin{enumerate}[I]%for capital roman numbers.
\item Basic properties of Orlicz Spaces, embedding theorems and continuity of nonlinear superposition (Nemytskii) operators.
\item Integral functionals in Orlicz spaces, lower semicontinuity and existence of minimizers.
\item $\Gamma-$convergence, integral representation, uniform estimates and compactness of $\Gamma-$limits.
\item Random functionals and ergodic theory.
\item Application of ergodic theorem to homogenization of equations with generalized growth conditions. Discussion on improvement of results for equations defined in Sobolev spaces with variable exponent ($p-$Laplace equations).
\end{enumerate}

\section{Preliminaries}

A function $\Phi$ is called a Young function if it admits the presentation
\[\Phi(u)=\int_{0}^{u}\phi(t)dt\]
where $\phi(t)$, $t>0$ satisfies
\begin{enumerate}
\item $\phi(0)=0$
\item $\phi(t)>0$ for all $t\geq 0$
\item $\phi$ is nondecreasing, right continuous
\item $\displaystyle\lim_{t\rightarrow\infty}\phi(t)=\infty$
\item $u\phi(u)<a\Phi(u)$ for $a>1$ and for all $u\geq 0$
\end{enumerate}

The function $\Phi$ is called an $N-$ function. We say that $\Phi$ satisfies the $\Delta_{2}-$condition (or has the doubling property) if there exists $k>0$, and $l\geq 0$ such that
\[\Phi(2u)\leq k\Phi(u)\]
for all $u>l$.

Then $\Phi$ is nonnegative, continuous, strictly increasing, convex function on $[0,\infty)$. The complementary function $\Psi$ to $\Phi$ is defined by the formula
\[\Psi(v)=\max_{u>0}[uv-\Phi(u)]\]

The following Young's inequality holds:
\[uv\leq \Phi(u)+\Psi(v)\]

Let $\Omega$ be an open subset of $\mathbb{R}^{n}$. We define the Orlicz class  $\tilde{L}_{\Phi}(\Omega)$ as the set of all measurable functions $u$ such that
\[\rho(u,\Phi,\Omega)=\int_{\Omega}\Phi(|u(x)|)dx<\infty\]

By Young's inequality, the norm
\[\||u\||_{\Phi,\Omega}=\sup_{\rho(v,\Phi,\Omega)\leq 1}\left|\int_{\Omega}u(x)v(x)dx\right|\]
is well defined in $\tilde{L}_{\Phi}(\Omega)$. If both $\Phi$, $\Psi$ satisfy the $\Delta_{2}-$condition, then the class $\tilde{L}_{\Phi}(\Omega)$ equipped with either the previous norm or the Luxenburg norm

\[\|u\|_{\Phi,\Omega}=\inf\left\{\lambda>0:\int_{\Omega}\Phi\left(\frac{|u(x)|}{\lambda}\right)~dx\leq 1\right\}\]
is a reflexive Banach space that we will denote by $L_{\Phi}(\Omega)$. For $u\in L_{\Phi}(\Omega)$, $v\in L_{\Psi}(\Omega)$, the following Holder inequality holds:
\[\int_{\Omega}uvdx\leq\|u\|_{\Phi,\Omega}\|v\|_{\Psi,\Omega}\]

Also,
\begin{equation}
\begin{array}{l}
\|u\|_{\Phi,\Omega}\leq \rho(u,\Phi,\Omega)+1
\end{array}
\label{eq:1}
\end{equation}
and if $\|u\|_{\Phi,\Omega}\leq 1$,
\begin{equation}
\begin{array}{l}
\rho(u,\Phi,\Omega)\leq \|u\|_{\Phi,\Omega}
\end{array}
\label{eq:2}
\end{equation}
We say that the sequence $u_{n}\in L_{\Phi}$ converges in the mean to $u$ provided
\[\rho(u_{n}-u,\Phi,\Omega)\rightarrow 0\]as $n\rightarrow\infty$. If the $\Delta_{2}$ condition holds, the convergence in the mean is equivalent to the convergence in norm.

\subsection{Orlicz-Sobolev spaces}

Consider the space of smooth functions $C^{1}(\Omega)$ endowed with the norm
\[\|u\|^{1}_{\Phi,\Omega}=\max_{|\alpha|\leq 1}\{\|D^{\alpha}u\|_{\Phi,\Omega}\}\]
The Orlicz-Sobolev space $W_{\Phi}^{1}(\Omega)$ is defined as the closure of $C^{1}(\bar{\Omega})$ with respect to this norm. The closure of $C_{0}^{\infty}(\Omega)$ to this norm is denoted by $W_{0,\Phi}^{1}(\Omega)$ and is a subspace of $W_{\Phi}^{1}(\Omega)$. The following embedding theorem holds \cite{Tru71}: if $\Omega$ is a bounded domain in $\mathbb{R}^{n}$ with smooth boundary, the embedding of $W_{\Phi}^{1}(\Omega)$ into $L_{\Phi}(\Omega)$ is compact and
\[\|u\|_{\Phi,\Omega}\leq C\|Du\|_{\Phi,\Omega}\]
for all $u\in W_{0,\Phi}^{1}(\Omega)$.

\begin{notation}
For two Young functions $Q$, $P$, we will use the symbol $P\prec\prec Q$ when $Q$ grows more rapidly than $P$ near infinity, i.e. for all $\delta>0$,
\[\lim_{t\rightarrow\infty}\frac{P(t)}{Q(\delta t)}=0\]
\end{notation}

\begin{lemma}[\cite{Krasno61},\cite{Benki99}]
Let $\Omega$ be an open subset of $\mathbb{R}^{n}$ of finite measure and let $\Phi,P$ be Young functions satisfying the $\Delta_{2}$ condition. Suppose that $g:\Omega\times\mathbb{R}\rightarrow\mathbb{R}$ is a Caratheodory function such that
\[P(|g(x,s)|)\leq k_{1}\Phi(|s|)\]
for some constant $k_{1}$. Then the Nemytskii operator $T_{g}(u)(x)=g(x,u(x))$ is strongly continuous from $L_{\Phi}(\Omega)$ to $L_{P}(\Omega)$.
\end{lemma}
For the proof we need the following theorem.
\begin{thm}
Suppose $\Phi\in\Delta_{2}$ and $\Omega\subset\mathbb{R}^{n}$. If $u_{n}\rightarrow u$ in $L_{\Phi}(\Omega)$ there exists a subsequence $u_{n_{k}}$ and $h\in L_{\Phi}(\Omega)$ such that
\begin{displaymath}u_{n_{k}}\rightarrow u\mbox{ a.e. in}\Omega \end{displaymath} and \begin{displaymath}|u_{n_{k}}|\leq h\mbox{ a.e. in}\Omega \end{displaymath}
\end{thm}
The proof of theorem $1$ is a modification of theorem $2.3$ in \cite{Ambro93}.

\begin{proof}[Proof of Lemma 1]
The map $u\rightarrow g(\cdot,u)$ is well defined from $L_{\Phi}(\Omega)$ to $L_{P}(\Omega)$. From theorem $1$ and the continuity of $g$ in $u$, for a sequence $u_{n}\rightarrow u$ in $L_{\Phi}(\Omega)$ we have
\[g(\cdot,u_{n})\rightarrow g(\cdot,u)\mbox{ a.e. in }\Omega\]
and
\[|u_{n}|\leq h\mbox{ a.e. in }\Omega\]
for all $n\in\mathbb{N}$ and $h\in L_{\Phi}(\Omega)$. The continuity of $P$ gives \begin{displaymath} P(|g(\cdot,u_{n})|)\rightarrow P(|g(\cdot,u)|)\mbox{ a.e. }\Omega\end{displaymath} Then
\[P(|g(\cdot,u_{n})|)\leq k_{1}\Phi(|u_{n}|)\leq k_{1}\Phi(|h|)\]
with $\Phi(|h|)\in L^{1}(\Omega)$. Thus, the Dominated convergence theorem says that
\[\int_{\Omega}P(|g(\cdot,u_{n})|)\rightarrow\int_{\Omega}P(|g(\cdot,u)|)\]
which implies that $g(\cdot,u_{n})\rightarrow g(\cdot,u)$ for any $u_{n}\rightarrow u$.
\end{proof}

\begin{lemma}[\cite{Adams75},8.23] Suppose $P$, $Q$ are Young functions with $P\prec\prec Q$. Then, any bounded subset of $L_{Q}(\Omega)$ which is precompact in $L^{1}(\Omega)$ is also precompact in $L_{P}(\Omega)$.
\end{lemma}

\begin{thm}[\cite{Adams75},8.32] Suppose $\Omega\subset\mathbb{R}^{n}$ has the cone property and $\displaystyle\int_{0}^{1}\frac{\Phi^{-1}(t)}{t^{n(n+1)}}~dt<\infty$ and $\displaystyle\int_{1}^{\infty}\frac{\Phi^{-1}(t)}{t^{n(n+1)}}~dt=\infty$. Consider the Young function defined by \[\Phi_{*}^{-1}(|t|)=\int_{1}^{|t|}\frac{\Phi^{-1}(s)}{s^{n(n+1)}}~ds\] Then for any $B\prec\prec\Phi_{*}$, the embedding $W_{\Phi}^{1}(\Omega)\rightarrow L_{B}(\Omega)$ is compact.
\end{thm}

\section{Integral functionals in Orlicz Spaces}
Let $f:\Omega\times\mathbb{R}^{n}\rightarrow\mathbb{\bar{R}}$ be a function such that
\begin{itemize}
\item $f(x,\cdot)$ is continuous a.e. $x\in\Omega$
\item $f(\cdot,p)$ is measurable for every $p\in\mathbb{R}^{n}$
\end{itemize}
We assume that
\begin{equation}
\begin{array}{l}
\text{f is convex in $p$};
\end{array}
\label{eq:3}
\end{equation}
and that there are constants $c^{1}_{f},c^{2}_{f}>0$ so that
\begin{equation}
\begin{array}{l}
c^{1}_{f}\Phi(|p|)\leq f(x,p)\leq c^{2}_{f}(1+\Phi(|p|))\text{ for all } (x,p)\in\mathbb{R}^{n}\times\mathbb{R}^{n}
\end{array}
\label{eq:4}
\end{equation}

Let also $g:\Omega\times\mathbb{R}\rightarrow\mathbb{\bar{R}}$ be a function that is lower semicontinuous in the second variable  and assume that for some constant $c_{g}>0$ and some $b_{g}(x)\in L^{1}(\Omega)$,

\begin{equation}
\begin{array}{l}
g(x,u)\geq c_{g}B(|u|)-b_{g}(x)
\end{array}
\label{eq:5}
\end{equation}
where $B\prec\prec\Phi_{*}$.
We denote by $\mathcal{F}=\mathcal{F}(c^{1}_{f},c^{2}_{f},\Phi)$ the class of functionals $F:L_{\Phi}(\mathbb{R}^{n})\times\Omega\rightarrow\bar{\mathbb{R}}$ such that
\begin{equation}
\begin{array}{l}
\displaystyle F(u,\Omega)=\int_{\Omega}f(x,D u(x))~dx\text{ for all }u\in W_{\Phi}^{1}(\Omega)
\end{array}
\label{eq:6}
\end{equation}

where $f$ is previously defined.

Let \begin{displaymath}G(u,\Omega)=F(u,\Omega)+\int_{\Omega}g(x,u(x))dx\end{displaymath}

\begin{proposition}
The functional $F(u,\Omega)$ is lower semicontinuous in the weak topology of $W_{\Phi}^{1}(\Omega)$, that is, if $u_{n}\rightarrow u$ weakly in $W_{\Phi}^{1}(\Omega)$, then
\[F(u,\Omega)\leq\liminf_{n\rightarrow\infty}F(u_{n},\Omega)\]
\end{proposition}
\begin{proof}
Let $v_{n}$ be a sequence converging to $v$ in $L_{\Phi}(\Omega)$ so that $\lim_{n\rightarrow\infty}F(v_{n},\Omega)$ exists. From the definition of the norm in $L_{\Phi}$ and Fatou's lemma,
\[\int_{\Omega}\Phi(\frac{u}{\|u\|})dx\leq 1\]
Hence,
\[\int_{\Omega}\liminf_{n\rightarrow\infty}\Phi(\frac{v_{n}-v}{\|v_{n}-v\|})dx\leq 1\]
The continuity of $\Phi$ implies that
\[\Phi\left[\liminf_{n\rightarrow\infty}(\frac{v_{n}-v}{\|v_{n}-v\|})\right]\leq\liminf_{n\rightarrow\infty}\Phi(\frac{v_{n}-v}{\|v_{n}-v\|})\leq\infty\text{ a.e.}\]
so that
\[\liminf_{n\rightarrow\infty}(\frac{v_{n}-v}{\|v_{n}-v\|})\leq\infty\text{ a.e.}\]
Then, there exists a subsequence, still denoted by $v_{n}$, such that $v_{n}\rightarrow v$ a.e.
Assume that $\|v_{n}-v\|<1/2$.
The convexity of $\Phi$ shows that
\begin{align*}
\Phi(v_{n})&=\Phi\left(\|v_{n}-v\|\frac{v_{n}-v}{\|v_{n}-v\|} +\left(1-\|v_{n}-v\| \right)\frac{v}{1-\|v_{n}-v\|}\right)
\\&\leq \|v_{n}-v\|\Phi\left(\frac{v_{n}-v}{\|v_{n}-v\|}\right)+\left(1-\|v_{n}-v\| \right)\Phi\left(\frac{v}{1-\|v_{n}-v\|}\right)
\end{align*}
Then,
\begin{equation}
\begin{array}{l}
\displaystyle\int_{\Omega}\Phi(v_{n})dx\leq\|v_{n}-v\|+(1-\|v_{n}-v\|)\int_{\Omega}\Phi(\frac{v}{1-\|v_{n}-v\|})dx
\end{array}
\label{eq:7}
\end{equation}

If in addition we choose $m$ such that $\|v\|\leq 2^{m}$, the $\Delta_{2}-$ condition reads
\[\Phi(\frac{v}{1-\|v_{n}-v\|})\leq\Phi(2v)\leq k\Phi(v)\]
and
\begin{align*}
\int_{\Omega}\Phi(v)dx&=\int_{\Omega}\Phi(\|v\|\frac{v}{\|v\|})dx
\\&\leq\int_{\Omega}\Phi(2^{m}\frac{v}{\|v\|})dx\leq k^{m}\int_{\Omega}\Phi(\frac{v}{\|v\|})dx\leq k^{m}<\infty
\end{align*}
From the Dominated convergence theorem,
\[\lim_{n\rightarrow\infty}\int_{\Omega}\Phi(\frac{v}{1-\|v_{n}-v\|})dx=\int_{\Omega}\Phi(v)dx\]
so that $(3.4)$ gives
\[\limsup_{n\rightarrow\infty}\int_{\Omega}\Phi(v_{n})dx\leq\int_{\Omega}\Phi(v)dx\]
Using again Fatou's lemma,
\[\int_{\Omega}\Phi(v_{n})dx\leq\liminf_{n\rightarrow\infty}\int_{\Omega}\Phi(v_{n})dx\]
so that
\[\int_{\Omega}\Phi(v_{n})dx\rightarrow\int_{\Omega}\Phi(v)dx\]

Since $f(x,\cdot)$ and $\Phi$ are continuous,
\[\lim_{n\rightarrow\infty}(f(x,v_{n})-c_{f}^{1}\Phi(v_{n}))=f(x,v)-c_{f}^{1}\Phi(v)\]
and
\[\int_{\Omega}f(x,v)dx\leq\lim_{n\rightarrow\infty}f(x,v_{n})dx\]
Thus, the functional $\displaystyle\int_{\Omega}f(x,v)dx$ is lower semicontinuous in $L_{\Phi}(\Omega)$ which implies that $F(u,\Omega)$ is lower semicontinuous in $W_{\Phi}^{1}(\Omega)$, since if $v_{n}\rightarrow v$ in $W_{\Phi}^{1}(\Omega)$, then $Dv_{n}\rightarrow Dv$ in $L_{\Phi}(\Omega)$.
\end{proof}

\begin{proposition}
The functional $\displaystyle\int_{\Omega}g(x,u)dx$ is sequentially lower semicontinuous in $W_{\Phi}^{1}(\Omega)$.
\end{proposition}

\begin{proof}

Suppose $u_{n}\rightarrow u$ weak $W_{\Phi}^{1}(\Omega)$. Then $\{u_{n}\}$ is bounded and up to a subsequence it converges strongly to $u$ in $L_{B}(\Omega)$, due to theorem $2$. Using proposition $1$ with g instead of $f$, we obtain that $\displaystyle\int_{\Omega}g(x,u)dx$ is lower semicontinuous on $L_{\Phi}(\Omega)$. This implies the needed result.
\end{proof}

\begin{thm}
Suppose $X$ is a nonempty, weakly closed subset of $W_{\Phi}^{1}(\Omega)$. Then the functional
\[G(u,\Omega)=F(u,\Omega)+\int_{\Omega}g(x,u)~dx\]
has a minimum over all  $u\in X$.
\end{thm}

\begin{proof}
Let $\chi_{X}$ be the indicator function of $X$, which is weakly lower semicontinuous in $W_{\Phi}^{1}(\Omega)$. Then the minimization problem can be written in the equivalent form
\[\min_{u\in W_{\Phi}^{1}(\Omega)}(F(u,\Omega)+\int_{\Omega}g(x,u)~dx+\chi_{X}(u))\]
Then,
\[F(u,\Omega)+\int_{\Omega}g(x,u)~dx+\chi_{X}\geq cZ(u)-b\]
holds for some positive constants $c,b$ with $Z(u)=\int_{\Omega}\Phi(u)~dx$ which is sequentially coercive in $W_{\Phi}^{1}(\Omega)$. The reason is that if $\|u\|\leq 1$ then $Z(u)\leq\|u\|$ and if $\|u\|\geq 1$ then $Z(u)\geq\|u\|$. Thus, the set $\{u:Z(u)\leq t\}$ is bounded in $W_{\Phi}^{1}(\Omega)$ and sequentially compact since the space is reflexive. The direct method of variational problems implies the existence of the minimizer. If $X$ is convex, one can show that the minimum is unique.
\end{proof}

\subsection{Yosida transforms and distance in $\mathcal{F}$}

For $F\in\mathcal{F}$ the $\varepsilon-$Yosida transform is the functional $T_{\varepsilon}F(u,\Omega):L_{\Phi}(\Omega)\times\Omega\rightarrow\bar{\mathbb{R}}$ defined by
\[T_{\varepsilon}F(u,\Omega)=\inf_{v\in W_{\Phi}^{1}(\Omega)}\{F(v,\Omega)+\varepsilon^{-1}\|u-v\|_{L_{\Phi}(\Omega)}\}\]

\begin{proposition}
For every $F\in\mathcal{F}$ and $u\in L_{\Phi}(\Omega)$,
\[\lim_{\varepsilon\rightarrow 0^{+}}T_{\varepsilon}F(u,\Omega)=\sup_{\varepsilon>0}T_{\varepsilon}F(u,\Omega)=F(u,\Omega)\]
\end{proposition}
\begin{proof}
See \cite{DalMod86}, proposition $1.11$
\end{proof}

The $\varepsilon-$Yosida transform can be used to define a metric in $\mathcal{F}$ so that the metric space $(\mathcal{F},d)$  is compact and the map $F\rightarrow\min_{u\in X}F(u)$ is continuous with respect to the metric. For this purpose, we pick a countable dense subset $W=\{w_{j}\}$ of $W_{\Phi}^{1}(\Omega)$ and a family $\mathcal{B}=\{B_{k}\}$ of open bounded subsets of $\mathbb{R}^{n}$.

Let $F$, $G\in\mathcal{F}$ and $h:\bar{\mathbb{R}}\rightarrow\mathbb{R}$, we define
\begin{equation}
\begin{array}{l}
\displaystyle d(F,G)=\sum_{i,j,k=1}^{\infty}\frac{1}{2^{i+j+k}}|h(T_{1/i}(F(w_{j},B_{k})))-h(T_{1/i}(G(w_{j},B_{k})))|
\end{array}
\label{eq:8}
\end{equation}

To show that $d$ is a distance in $\mathcal{F}$, it suffices to show that if $d(F,G)=0$, then $F=G$.

\section{$\Gamma-$convergence in Orlicz Spaces}
The next step is to show that $(\mathcal{F},d)$ is compact, thus separable and complete. The notion of $\Gamma-$convergence will be introduced for this purpose.

\begin{definition}
Let $X$ be a metric space and $F_{n}:X\rightarrow\bar{\mathbb{R}}$ a sequence of functionals on $X$. We say that $F_{n}$ $\Gamma(X)-$converges to the $\Gamma(X)-$limit $ F:X\rightarrow\bar{\mathbb{R}}$ if the following two conditions hold:
\begin{itemize}
\item $\displaystyle F(x)\leq\liminf_{n\rightarrow\infty} F_{n}(x_{n})$, for every sequence $x_{n}$ converging to $x$ as  $n\rightarrow\infty$
\item for every $x\in X$, there is a sequence $x_{n}$ converging to $x$ as  $n\rightarrow\infty$ such that $\displaystyle F(x)\geq\limsup_{n\rightarrow\infty} F_{n}(x_{n})$
\end{itemize}
In this case we write $\displaystyle F(x)=\Gamma(X)\lim_{n\rightarrow\infty}F_{n}(x)$.
\end{definition}

We adopt the definition of $\Gamma-$convergence for functionals in $\mathcal{F}$ and $u\in L_{\Phi}(\Omega)$ and we denote the limit by
\[\Gamma(L_{\Phi})\lim_{n\rightarrow\infty}F_{n}(u)=F(u)\]

\begin{proposition}[Main $\Gamma-$Convergence result]
The class $\mathcal{F}$ is compact for the $\Gamma(L_{\Phi})$ convergence, i.e. every sequence $\{F_{n}\}$ in $\mathcal{F}$ contains a subsequence that $\Gamma(L_{\Phi})-$converges to a functional $F\in\mathcal{F}$.
\end{proposition}

The proof of proposition $4$ will be a consequence of the following results:
\begin{proposition}
Let $(X,d)$ be a separable metric space, and for all $j\in\mathbb{N}$ let $f_{j}:X\rightarrow\bar{\mathbb{R}}$ be a function. Then there is an increasing sequence of integers $(j_{k})$ such that the $\displaystyle\Gamma(d)-\lim_{k}f_{j_{k}}$ exists for all $x\in X$.
\end{proposition}
\begin{proof}
See \cite{gammacon}, chap. $8$.
\end{proof}
Note that if the $\Delta_{2}-$condition is satisfied the Orlicz space is separable \textsc{}\cite{Kufner77}. In the following definition, $\mathcal{A}(A)$ is the family of all open subsets of $A\subset\mathbb{R}^{n}$.
\begin{definition}
A function $\alpha:\mathcal{A}(\Omega)\rightarrow [0,\infty]$ is called an increasing set function if $\alpha(\emptyset)=0$ and $\alpha(A)\leq\alpha(B)$ if $A\subset B$. An increasing set function is subadditive if $\alpha(A\cup B)\leq\alpha(A)+\alpha(B)$ for all $A,B\subset\mathcal{A}(\Omega)$. Finally, $\alpha$ is called inner regular if
\[\alpha(A)=\sup\{\alpha(B)|B\in A, B\subset\subset A\}\]
\end{definition}

\subsection{An Integral representation for $\Gamma-$limits}
We recall that a function $u\in L^{1}(\Omega)$ is piecewise affine in $\Omega$ if there is a family of disjoint open subsets of $\Omega$ and a set $N\subset\Omega$ with $|N|=0$ such that $\Omega=\left(\bigcup_{i\in I}\Omega_{i}\right)\bigcup N$ and $u_{|\Omega_{i}}$ is affine in $\Omega_{i}$. We have the following density result.

\begin{proposition}
For every $u\in W_{\Phi}^{1}(\Omega)$ there exists a sequence $u_{j}\in W_{\Phi}^{1}(\Omega)$
 of piecewise affine functions such that $u_{j}\rightarrow u$ in $W_{\Phi}^{1}(\Omega)$.
\end{proposition}

\begin{proof}
Applying theorem $2.1$ of \cite{Tru71}, we can find a sequence $\{u_{j}\}\in C_{0}^{\infty}(\Omega)$ converging to $u$ in $W_{\Phi}^{1}(\Omega)$. Furthermore, by proposition $2.1$, chap. $X$ of \cite{Ekeland76}, for $u\in C_{0}^{\infty}(\Omega)$ there is a sequence $\{u_{j}\}$ of piecewise affine functions so that $u_{j}\rightarrow u$ and $Du_{j}\rightarrow Du$ uniformly in $\Omega$. Then, since $u_{j}\in W_{\Phi}^{1}(\Omega)$ and the uniform convergence implies that
\[\int_{\Omega}\Phi(|u_{j}-u|)+\Phi(|Du_{j}-Du|)~dx\rightarrow 0\]
a diagonal process gives the desired sequence.
\end{proof}

We are in position to state and show an integral representation result for a class of functionals in the Orlicz space.
\begin{thm}
Suppose that $F:L_{\Phi}(\Omega)\times\mathcal{A}\rightarrow [0,\infty)$ be an increasing functional satisfying the following assumptions:
\begin{enumerate}
\item $F$ is local, i.e. $F(u,A)=F(v,A)$ for all $A\in\mathcal{A}$ and $u,v$ such that $u=v$ a.e. in $A$;
\item $F$ is lower semicontinuous;
\item $F(u+c,A)=F(u,A)$ for every $u\in L_{\Phi}(\Omega)$, $c\in\mathbb{R}^{n}$
\item there is constant $\beta>0$ and a function $a(x)\in L^{1}(\Omega)$ such that
\[0\leq F(u,A)\leq\beta\int_{A}a(x)+\Phi(|Du|)~dx\]
for all $u\in W_{\Phi}^{1}(\Omega)$, $A\in\mathcal{A}(\Omega)$
\end{enumerate}
Then there exist a Caratheodory function $f:\Omega\times\mathbb{R}^{n}\rightarrow [0,\infty]$ such that
\begin{enumerate}[(i)]
\item for every $u\in L_{\Phi}(\Omega)$,
\[F(u,A)=\int_{A}f(x,Du(x))~dx\]
\item $f(x,\cdot)$ is convex for every $x\in\Omega$ and it satisfies
\[0\leq f(x,p)\leq a(x)+\Phi(|p|)\]
\end{enumerate}
\end{thm}

\begin{proof}
The proof, in general, follows the steps with the proof in the case of Sobolev spaces \cite{gammacon} with a few differences.
\begin{itemize}
\item[\textbf{Step 1}]
We define the linear function $u_{p}(x)=p\cdot x$ and use assumption $(4)$ of the theorem to claim that $F(u_{p},\cdot)$ is continuous with respect to the Lebesgue measure. Thus there is a density function \begin{displaymath}f(x,p)=\lim_{\rho\rightarrow 0^{+}}\frac{F(u_{p},B_{\rho}(x))}{|B_{\rho}(x)|}\end{displaymath} in $L_{\operatorname*{loc}}^{1}(\Omega)$ such that
\[F(u_{p},A)=\int_{A}f(x,p)~dx\]
for $A\in\mathcal{A}$. One can show that the representation $(i)$ holds for every piecewise affine function.
\item[\textbf{Step 2}] It can be shown that $f(x,\cdot)$ is convex on $\mathbb{R}^{n}$:
\[f(x,p)\leq tf(x,p_{1})+(1-t)f(x,p_{2})\] for all $t\in [0,1]$, $p_{1}\neq p_{2}$ with $p=tp_{1}+(1-t)p_{2}$.
\item[\textbf{Step 3}] The map \begin{displaymath}u\rightarrow\int_{A}f(x,Du(x))~dx \end{displaymath} is continuous with respect to the $W_{\Phi}^{1}(\Omega)$ convergence. Let $u\in W_{\Phi}^{1}(\Omega)$ and $A\in\mathcal{A}$. Applying proposition $6$, for $A'\subset A$ there exists a sequence $u_{j}$ of piecewise affine functions such that $u_{j}\rightarrow u$ in $W_{\Phi}^{1}(\Omega)$. By the lower semicontinuity of $F$ we have that
\[F(u,A')\leq\liminf_{j\rightarrow\infty}F(u_{j},A')=\lim_{j\rightarrow\infty}\int_{A'}f(x,Du_{j})~dx=\int_{A'}f(x,Du)~dx\]
Taking the limit $A'\nearrow A$, we obtain $\displaystyle F(u,A)\leq\int_{A}f(x,Du)~dx$ for all $u\in W_{\Phi}^{1}(\Omega)$ and $A\in\mathcal{A}$.
\item[\textbf{Step 4}] Fix $v\in W_{\Phi}^{1}(\Omega)$ and define the functional
\[W(u,A)=F(u+v,A)\]
It is straighforward to see that $G$ satisfies assumptions $(1)-(3)$ and $(4)$ can ve verified with the computation
\begin{align*}0\leq W(u,A)=F(u+v,A)&\leq\int_{A}a(x)+\Phi(|Du +Dv|)~dx
\\&\leq 2^{\beta-1}\int_{A}\frac{a(x)}{2^{\beta-1}}+\Phi(|Du|) +\Phi(Dv|)~dx
\\&=\int_{A}b(x)+\Phi(|Du|)~dx
\end{align*}
where $b(x)=a(x)+2^{\beta-1}\Phi(Dv|)\in L^{1}(\Omega)$, $\beta=\log_{2}k$. Thus, from step $1$, we can find a measurable function $g:\Omega\times\mathbb{R}^{n}$ such that
\[W(u,A)\leq\int_{A}g(x,Du)~dx\]
for all piecewise affine functions $u\in W_{\Phi}^{1}(\Omega)$ and $A\in\mathcal{A}$. It follows that the map $u\rightarrow\int_{A}g(x,Du)~dx$ is continuous in $W_{\Phi}^{1}(\Omega)$. For $A'\subset A$ there is a sequence of piecewise affine functions $(v_{n})$ converging to u in $W_{\Phi}^{1}(\Omega)$ so that step $1$ together with the last inequality give us
\begin{align*}
\int_{A'}&g(x,0)~dx=W(0,A')=F(v,A')
\\&\leq\int_{A'}f(x,Dv)~dx=\lim_{n\rightarrow\infty}\int_{A'}f(x,Dv_{n})~dx=\lim_{n\rightarrow\infty}F(v_{n},A')
\\&\lim_{n\rightarrow\infty}W(v_{n}-v,A')\leq\lim_{n\rightarrow\infty}\int_{A'}g(x,Dv_{n}-Dv)~dx=\int_{A'}g(x,0)~dx
\end{align*}
so as $A'\nearrow A$ we get $\displaystyle F(v,A)=\int_{A}f(x,Dv)~dx$.
\end{itemize}
\end{proof}

\subsection{Uniform Estimate}
 To proceed to the compactness for integral functionals, we need to prove some properties of the $\Gamma-$limit as a set function. We do that by elaborating a method of joining sequences of functions so that, from the knowledge of the minimizing sequences for $F(u,A)$ and $F(u,B)$, we can obtain an estimate for $F(u,A\cup B)$. We need the following lemma.

\begin{lemma}
Suppose $U,U',V\in\mathcal{A}(A)$ with $U'\subset\subset U$ and let $u\in W_{\Phi}^{1}(U)$, $v\in W_{\Phi}^{1}(V)$. Then for every cutoff function $\phi\in C_{0}^{\infty}(\Omega)$ between $U'$ and $U$ (i.e. $\operatorname*{spt}\phi\subset U$, $0\leq\phi\leq 1$, $\phi=1$ in $U'$) we have $\phi u+(1-\phi)v\in W_{\Phi}^{1}(U'\cup V)$.
\end{lemma}

\begin{definition}
Let $F:L_{\Phi}\times\Omega\rightarrow\bar{\mathbb{R}}$ be a functional. We say that $F$ satisfies the fundamental $L_{\Phi}-$estimate if for every $U,U',V\in\mathcal{A}(A)$ with $U'\subset\subset U$ and $\sigma>0$ there is $M_{\sigma}>0$ such that for all $u,v\in L_{\Phi}(\Omega)$ there exists a cut-off function $\phi$ between $U'$ and $U$ such that
\begin{multline}\label{eq:9} F(\phi u+(1-\phi)v,U'\cup V)\\
\leq (1+\sigma)(F(u,U)+F(v,V))+M_{\sigma}\int_{(U\cap V)\setminus U')}\Phi(|u-v|)~dx+\sigma
\end{multline}

The same definition holds for a family $\{F_{n}\}_{n>0}$ if in addition there is $n_{0}$ such that for all $n\leq n_{0}$ the last estimate is valid uniformly.
\end{definition}

This definition corresponds to the definition of $L^{p}-$fundamental estimate for functionals defined in $L^{p}$ spaces.

\begin{proposition}[Uniform estimate]
The family $\mathcal{F}(c^{1}_{f},c^{2}_{f},\Phi)$ satisfies the fundamental $L_{\Phi}-$estimate uniformly.
\end{proposition}

\begin{proof}
Pick $F\in\mathcal{F}(c^{1}_{f},c^{2}_{f},\Phi)$ and let $U,U',V\in\mathcal{A}(A)$ with $U'\subset\subset U$. Let $\delta=d(U',\partial U)$ and take parameters $0<\eta<\delta$, $0<r<\delta-\eta$. Choose a cutoff function $\phi$ between the sets $\{x\in U:d(x,U')<r\}$ and $\{x\in U:d(x,U')<r+\eta\}$ with $|D\phi|\leq 2/\eta$. Define $V_{r}^{\eta}=\{x\in V:r<d(x,U')<r+\eta\}$.

Let $u,v\in L_{\Phi}(\Omega)$. Then
{\allowdisplaybreaks
\begin{align*}
F(u\phi&+(1-\phi)v,U'\cup V)
\\&=\int_{U'\cup V}f(x,\phi Du+(1-\phi)Dv+(u-v)D\phi)~dx
\\&=\int_{\{x\in V:d(x,U')\geq r+\eta\}}f(x,Dv)~dx+\int_{\{x\in V\cup U':d(x,U')\leq r\}}f(x,Du)~dx
\\&+\int_{V_{r}^{\eta}}f(x,\phi Du+(1-\phi)Dv+(u-v)D\phi)~dx
\\&\leq F(v,V)+F(u,U)+c^{2}_{f}\int_{V_{r}^{\eta}}1+\Phi(|\phi Du+(1-\phi)Dv+(u-v)D\phi|)~dx
\\&= F(v,V)+F(u,U)+c^{2}_{f}\int_{V_{r}^{\eta}}1+\Phi\left(3\frac{|\phi Du+(1-\phi)Dv+(u-v)D\phi|}{3}\right)~dx
\\&\leq F(v,V)+F(u,U)
\\&+k3^{\beta}c^{2}_{f}\int_{V_{r}^{\eta}}1+\Phi\left(\frac{|\phi Du+(1-\phi)Dv+(u-v)D\phi|}{3}\right)~dx
\\&\leq F(v,V)+F(u,U)
\\&+k3^{\beta-1}c^{2}_{f}\int_{V_{r}^{\eta}}1+\Phi(|\phi Du|)+\Phi(|(1-\phi)||Dv|)+\Phi(|(u-v)||D\phi|)~dx
\\&\leq F(v,V)+F(u,U)
\\&+k3^{\beta-1}c^{2}_{f}\int_{V_{r}^{\eta}}1+\Phi(|Du|)+\Phi(|Dv|)~dx+k3^{\beta-1}c^{2}_{f}(\frac{2}{\eta})^{\beta}\int_{(U\cap V)\setminus U'}\Phi(|(u-v)|)~dx
\end{align*}
}
where $\beta=\log_{2}k$. To obtain the above inequalities we use assumption $(3.2)$, the properties of $\phi$, Jensen's inequality and the fact that the $\Delta_{2}-$condition implies that $\Phi(\lambda u)\leq k\lambda^{\beta}\Phi(u)$, for $\lambda\geq 1$.
Note that from $(3.2)$,

\[c^{2}_{f}\int_{U\cap V}1+\Phi(|Du|)+\Phi(|Dv|)~dx\leq c^{2}_{f}|U\cap V|+\frac{c^{2}_{f}}{c^{1}_{f}}(F(u,U)+F(v,V))\]

Then, for all $N=1,2,...$ there is $\mu\in\{1,,,,N\}$  such that
\begin{align*}c^{2}_{f}k3^{\beta-1}&\int_{\{x\in V:\frac{\delta(\mu-1)}{N}<d(x,U')<\frac{\delta \mu}{N}\}}1+\Phi(|Du|)+\Phi(|Dv|)~dx
\\&\leq c^{2}_{f}k3^{\beta-1}\frac{1}{N}|U\cap V|+c^{2}_{f}k3^{\beta-1}\frac{1}{N}\frac{c^{2}_{f}}{c^{1}_{f}}(F(u,U)+F(v,V))
\end{align*}

Fix $\sigma$ and choose $\displaystyle N\geq\max\left\{\frac{c^{2}_{f}k3^{\beta-1}}{\sigma},\frac{1}{\sigma}k3^{\beta-1}\frac{c^{2}_{f}}{c^{1}_{f}}\right\}$, $\displaystyle \eta=\frac{\delta}{N}$ and $\displaystyle r=\frac{(\mu-1)\delta}{N}$ so that the constant $M_{\sigma}$ depends only on $U,U',V,c^{1}_{f},c^{2}_{f}$. This implies that the estimate holds uniformly in $\mathcal{F}(c^{1}_{f},c^{2}_{f},\Phi)$.
\end{proof}

Using the uniform estimate, one can include boundary conditions to the study of $\Gamma-$limits of local functionals.

\begin{proposition}
Suppose that $\{F_{n}\}$ is a family of functionals defined on $L_{\Phi}(\Omega)\times\mathcal{A}(A)$ that satisfy the fundamental $L_{\Phi}$ estimate as $n\rightarrow 0$ and let $(n_{j})$ be a sequence of positive numbers converging to zero. If
$F'(u,U)=\Gamma(L_{\Phi})-\liminf_{j}F_{n_{j}}(u,U)$ and
$F''(u,U)=\Gamma(L_{\Phi})-\limsup_{j}F_{n_{j}}(u,U)$ then
\[F'(u,U'\cup V)\leq F'(u,U')+F''(u,V)\] and
\[F''(u,U'\cup V)\leq F''(u,U)+F''(u,V)\] for all $u\in L_{\Phi}(\Omega)$ and $U,U',V\in \mathcal{A}(A)$.
\end{proposition}

\begin{proof}
We start by noticing that we can find two sequences $\{u_{j}\}$ and $\{v_{j}\}$ converging to $u$ in $L_{\Phi}$ such that
\[F'(u,V)=\liminf_{j}F_{n_{j}}(u_{j},V)\] and  \[F''(u,V)=\limsup_{j}F_{n_{j}}(v_{j},V)\]
Applying the fundamental estimate to $\{u_{j}\}$ and $\{v_{j}\}$ for fixed $\sigma>0$, we can find $M_{\sigma},n_{\sigma}>0$ so that for $n_{j}<n_{\sigma}$ there is a sequence $w_{j}=\phi_{j}u_{j}+(1-\phi_{j})v_{j}$ ($\phi_{j}$ are cutoff functions) between $U',U$ such that
\[F_{n_{j}}(w_{j},U'\cup V)\leq (1+\sigma)(F_{n_{j}}(u_{j},U)+F_{n_{j}}(v_{j},V))+M_{\sigma}\int_{(U\cap V)}\Phi(|u_{j}-v_{j}|)~dx+\sigma\]
Since $\displaystyle\int_{U\cap V}\Phi(|u_{j}-v_{j}|)~dx\rightarrow 0$, we have
\begin{align*}F'(u,U\cup V)\leq\liminf_{j}&F_{n_{j}}(w_{j},U'\cup V)
\\&\leq(1+\sigma)(\liminf_{j}F_{n_{j}}(u_{j},U)+\limsup_{j}F_{n_{j}}(v_{j},V))+\sigma
\\&=(1+\sigma)(F'(u,U)+F''(u,V))+\sigma
\end{align*}
for any $\sigma>0$. The proof of the second inequality follows the same steps.
\end{proof}

Then one can show that $F'(u,\cdot),F''(u,\cdot)$ are increasing set functions. We return to the proof of proposition $4$:
\begin{proof}[Proof of proposition 4]
Taking into account the compactness result $10.3$ in \cite{Brai98} and the above propositions, there exists a subsequence $F_{h(n)}$ and a nonnegative, convex function $f_{\gamma}:\mathbb{R}^{n}\times\mathbb{R}^{n}\rightarrow\mathbb{R}$ such that
\[\int_{\Omega}f_{\gamma}(x,Du)~dx=\Gamma(L_{\Phi})\lim_{h\rightarrow\infty}(F_{h(n)})\] for each $\Omega\in\mathcal{A}(A)$, $u\in L_{\phi}(\Omega)$. The growth conditions of $f_{\gamma}$ remain the same (because of the lower semicontinuity) so that $f_{\gamma}\in\mathcal{F}(c^{1}_{f},c^{2}_{f},\Phi)$.
\end{proof}

\begin{proposition}
Let $\Omega\in\mathcal{A}$ and $F_{n}$ be a sequence in $\mathcal{F}$. Suppose that $X$ is a weakly closed subset in $W_{\Phi}^{1}(\Omega)$. Suppose $\{F_{n}(u,\Omega)\}$ $\Gamma (L_{\Phi})-$ converges to a functional $\tilde{F}\in\mathcal{F}$. Let $X$ be a weakly closed subspace of $W_{0,\Phi}^{1}(\Omega)$ and
\[c_{1}(x)+P^{-1}c_{g}^{1}B(|u|)\leq g(x,u)\leq c_{2}(x)+P^{-1}c_{g}^{2}B(|u|)\]
where $c_{i}(x)\in L^{1}(\Omega)$ and $\Phi\prec\prec P$. Then,
\[\lim_{n\rightarrow\infty}\min_{u\in X}G_{n}(u)=\min_{u\in X}\tilde{G}(u)\]
where $\tilde{G}(u)=\tilde{F}(u,\Omega)+\int_{\Omega}g(x,u(x))dx$.
Furthermore, any sequence $u_{n}\in X$ with $\displaystyle G_{n}(u_{n})=\min_{u\in X}G_{n}(u)$ contains a subsequence that converges strongly in $L_{\Phi}(\Omega)$, weakly in $W_{\Phi}^{1}(\Omega)$ and a.e. in $\Omega$ to a function $\tilde{u}\in X$ such that
\[\tilde{G}(\tilde{u})=\min_{u\in X}\tilde{G}(u)\]
\end{proposition}

\begin{proof}
From theorem $3$, the functionals $G_{n},\tilde{G}$ attain their minimum in $X$. For any sequence such that $\displaystyle G_{n}(u_{n})=\min_{u\in X}G_{n}(u)$, we have that $u_{n}$ is bounded in $W_{\Phi}^{1}(\Omega)$ so, up to a subsequence $u_{n_{k}}$, it converges in $L_{\Phi}(\Omega)$ and pointwise a.e. to a function $\tilde{u}\in W_{\Phi}^{1}(\Omega)$ and
\begin{equation}
\begin{array}{l}
\displaystyle\liminf_{n\rightarrow\infty}\left(\min_{u\in X}G_{n}(u)\right)=\liminf_{n\rightarrow\infty}\left(\min_{u\in X}G_{n_{k}}(u)\right)
\end{array}
\label{eq:10}
\end{equation}
Take a fixed subset $B\subset\subset\Omega$; if the sequence $u_{n_{k}}$ converges to $\tilde{u}$, the sequence $F_{n_{k}}$ $\Gamma(L_{\Phi})-$converges to $\tilde{F}_{n_{k}}$ in $B$ and hence
\[\tilde{F}(\tilde{u},B)\leq\liminf_{n\rightarrow\infty}F_{n_{k}}(u_{n_{k}},B)\leq\liminf_{n\rightarrow\infty}F_{n_{k}}(u_{n_{k}},\Omega)\]
 Fatou's lemma gives
\[\liminf_{n\rightarrow\infty}\int_{\Omega}g(x,u_{n_{k}})~dx\geq\int_{\Omega}g(x,\tilde{u})~dx \]
Combining, we obtain
\begin{align*}\tilde{F}(\tilde{u},B)+&\int_{\Omega}g(x,\tilde{u})~dx
\\&\leq\liminf_{n\rightarrow\infty}\left[F_{n_{k}}(u_{n_{k}},B)+\int_{\Omega}g(x,u_{n_{k}})~dx\right]
=\liminf_{n\rightarrow\infty}G_{n_{k}}(u_{n_{k}})
\end{align*}
and hence taking $B\uparrow\Omega$ yields
\[\tilde{G}(\tilde{u})\leq\liminf_{n\rightarrow\infty}G_{n_{k}}(u_{n_{k}})\]
Consider now the second term $\displaystyle K(u)=\int_{\Omega}g(x,u)~dx$ and note that Lemma $1$ implies the continuity of $K$ in X.

We will show that for all $\varepsilon>0$ there is a sequence $(v_{n})$ in $X$ converging to $\tilde{v}$ so that
\[\limsup_{n\rightarrow\infty}F_{n}(v_{n},\Omega)\leq (1+\varepsilon)\tilde{F}(\tilde{v},\Omega)+c\varepsilon\]
with $c=c(\tilde{v})$. Fix $\varepsilon\in [0,1]$. Our assumption says that there is a sequence  $(w_{n})$ converging to $\tilde{v}$ in $L_{\Phi}(\Omega)$ such that
\[\tilde{F}(\tilde{v},\Omega)\geq\limsup_{n\rightarrow\infty}F_{n}(w_{n},\Omega)\]
To have that $v_{n}\in X$, we modify $w_{n}$ by taking a compact subset $B$ of $\Omega$ with
\[\int_{\Omega\setminus B}(1+\Phi(|D\tilde{v}|))~dx<\varepsilon\]
and sets $\Omega_{1},\Omega_{2}$ with $B\subset\Omega_{1}\subset\Omega_{2}\subset\Omega$. Applying proposition $6$, one can find $M>0$ and cutoff functions $\phi_{1},\phi_{2},...,\phi_{k}$ of $C_{0}^{\infty}(\Omega_{2})$ between the sets $\Omega_{1},\Omega_{2}$ so that
\begin{align*}\min_{1\leq i\leq k}F_{n}(\phi_{i}w_{n}&+(1-\phi_{i})\tilde{v},\Omega)
\\&\leq (1+\varepsilon)[F_{n}(w_{n},\Omega)+F_{n}(\tilde{v},\Omega\setminus B)]
\\&+\varepsilon\left[\int_{\Omega}\Phi(w_{n})~dx+\int_{\Omega\setminus B}\Phi(\tilde{v})~dx+1\right]+M\int_{\Omega\setminus B}\Phi(w_{n}-\tilde{v})~dx
\end{align*}
for $n\in\mathbb{N}$. If we denote by $i_{n}$ the index at which the minimum is attained and define $v_{n}=\phi_{i_{n}}w_{n}+(1-\phi_{i_{n}})\tilde{v}\in X$ then $v_{n}$ converges to $\tilde{v}$ in $X$. Finally,
\begin{align*}\limsup_{n\rightarrow\infty}&F_{n}(v_{n},\Omega)
\\&\leq (1+\varepsilon)\left[\limsup_{n\rightarrow\infty}F_{n}(w_{n},\Omega)+C\int_{\Omega\setminus B}(1+\Phi(|D\tilde{v}|))~dx\right]
\\&+\varepsilon\left[2\int_{\Omega}\Phi(\tilde{v})~dx+1\right]
\\&\leq (1+\varepsilon)\tilde{F}(\tilde{v},\Omega)+c\varepsilon
\end{align*}
which completes the proof.

\end{proof}

In our previous discussion, the function $f$ may be taken in the form $f(x,z,p)$ instead of $f(x,p)$.

The following theorem shows the connection between the distance $d$ and the $\Gamma-$convergence. Together with proposition $4$, it shows that $(\mathcal{F},d)$ is complete.

\begin{thm}
Suppose that $\{F_{n}\}_{n>0}$ is a sequence in $\mathcal{F}$ and $F_{\gamma}\in\mathcal{F}$. The following conditions are equivalent:
\begin{enumerate}
\item $\displaystyle\lim_{n\rightarrow\infty}d(F_{n},F_{\gamma})=0$
\item $\displaystyle\Gamma(L_{\Phi})\lim_{n\rightarrow\infty}(F_{n})=F_{\gamma}$
\item $\displaystyle\lim_{n\rightarrow\infty}(T_{\varepsilon}F_{n})(u,\Omega)=(T_{\varepsilon}F_{\gamma})(u,\Omega)$ $\forall\varepsilon>0$, $u\in L_{\Phi}$, $\Omega\in\mathcal{A}(A)$
\end{enumerate}
\end{thm}

\begin{corollary}
The map
$\displaystyle m_{\Omega,X,g}(F)=\min_{u\in X} J(u,\Omega)$
is continuous on $(\mathcal{F},d)$ for $\Omega\in\mathcal{A}(A)$, $X\subseteq W_{\Phi}^{1}(\Omega)$.
\end{corollary}
\section{Random functionals and the ergodic theorem}

We denote by $(\mathcal{S},\Sigma,P)$ a fixed probability space where $\Sigma$ is the $\sigma-$algebra on $\mathcal{S}$ and $P$ is the probability measure. A random functional is a measurable function $F:\mathcal{S}\rightarrow\mathcal{F}$ when $\mathcal{F}$ is endowed with the field $\Sigma_{\mathcal{S}}$ generated by the distance $d$. The image $P(F^{-1}(S))$, $S\in\Sigma_{\mathcal{S}}$ is the distribution law of $F$.  If $F$ and $G$ have the same distribution law, we write $F\sim G$.

For $z\in\mathbb{Z}^{n}$ we define the translation operator $\tau_{z}$ by
\begin{equation}
\begin{array}{l}
\displaystyle(\tau_{z}F)(u,\Omega)=\int_{\Omega}f(x+z,Du)~dx
\end{array}
\label{eq:11}
\end{equation}
and for $\varepsilon>0$ the homothety operator $\rho_{\varepsilon}$ by
\begin{equation}
\begin{array}{l}
\displaystyle(\rho_{\varepsilon}F)(u,\Omega)=\int_{\Omega}f(\frac{x}{\varepsilon},Du)~dx
\end{array}
\label{eq:12}
\end{equation}

Note that, since the integrand is independent of $u$,
\[(\tau_{z}F)(u,\Omega)=F(\tau_{z}u,\tau_{z}\Omega)\]
where $\tau_{z}u(x)=(x-z)$, $\tau_{z}\Omega=\{x\in\mathbb{R}^{n}:x-z\in\Omega\}$ and

\[\rho_{\varepsilon}F(u,\Omega)=\varepsilon^{n}F(\rho_{\varepsilon} u,\rho_{\varepsilon}\Omega)\]
where $\rho_{\varepsilon} u=\frac{1}{\varepsilon}u(\varepsilon x)$, $\rho_{\varepsilon}\Omega=\{x\in\mathbb{R}^{n}:\varepsilon x\in\Omega\}$. If $F$ is a random functional then both the translated and the homothetic functionals are also random functionals.

A stochastic homogenization process is a family of random variables $(F_{\varepsilon})_{\varepsilon>0}$ that has the same distribution law with the random functionals $\rho_{\varepsilon}F$.

For $F\in\mathcal{F}$, $u\in L_{\Phi}(\Omega)$, we consider the Dirichlet problem

\begin{equation}
\begin{array}{l}
\displaystyle m(F,u_{0},\Omega)=\min_{u}\{F(u,\Omega):u-u_{0}\in W_{0,\Phi}^{1}(\Omega)\}
\end{array}
\label{eq:13}
\end{equation}

Let $Q_{1/\varepsilon}$ be the cube
\[Q_{1/\varepsilon}=\{x\in\mathbb{R}^{n}:|x_{i}|<1/\varepsilon,i=1,..,n\}\]
with volume $(2/\varepsilon)^{n}$. We denote by $l_{p}=p\cdot x$ the linear function with gradient $p$. The main theorem of this section is the following:

\begin{thm}
Let $F$ be a random integral functional and define $F_{\varepsilon}=\rho_{\varepsilon}F$. Suppose that $\tau_{z}F=F(\tau_{z}u,\tau_{z}A)$ and $F$ have the same distribution law. Then the family $F_{\varepsilon}$ converges $P-$almost everywhere as $\varepsilon\rightarrow 0^{+}$ to a random integral functional $F_{0}$. In addition, there is a set $\mathcal{S}'\subset\mathcal{S}$ of full measure such that the limit
\[f_{0}(\omega,p)=\lim_{\varepsilon\rightarrow 0^{+}}\frac{m(u,l_{p},Q_{1/\varepsilon})}{|Q_{1/\varepsilon}|}\]
exists for all $\omega\in\mathcal{S}'$ and
\[F_{0}(u,\Omega)=\int_{\Omega}f_{0}(\omega,p)~dx\]
Moreover, if $F$ is ergodic, the integrand $f_{0}$ is independent of $\omega$ and
\[f_{0}(p)=\lim_{\varepsilon\rightarrow 0^{+}}\int_{\Omega}\frac{m(u,l_{p},Q_{1/\varepsilon})}{|Q_{1/\varepsilon}|}\]
for all $p\in\mathbb{R}^{n}$, $\varepsilon>0$.
\end{thm}

We give a sketch of the proof which can be found in \cite{Dal86}. We need the following results.

In \cite{DaMod86}, the question of determining an integral functional by the knowledge of their minima was studied. In particular, suppose the numbers
\[m(u,l_{p},\Omega)\]
are given and that we have a family of subsets $\{A_{\rho}\}_{\rho>0}$ of $\mathbb{R}^{n}$ which shrinks nicely to $x$ as $\rho\rightarrow 0^{+}$. This means that the following density-type inequalities are satisfied:

\[A_{\rho}\subseteq B(x,\rho)\qquad |A_{\rho}|\geq c|B(x,\rho)|\]

where $B(x,\rho)$ is the ball centered at $x$ of radius $\rho$. The following theorem holds

\begin{thm}
Suppose $f:\Omega\times\mathbb{R}^{n}\rightarrow\mathbb{R}$ satisfies the assumptions of section $3$, $(3.1)$ and
\[\phi_{1}(p)\leq f(x,p)\leq\phi_{2}(p)\]
for all $(x,p)\in(\Omega\times\mathbb{R}^{n})$ where $\phi_{1}$, $\phi_{2}$ are convex in $p$ and
\[\lim_{|p|\rightarrow\infty}\frac{\phi_{1}(p)}{|p|}=\infty\]
Then there is a measurable subset $N\subseteq\mathbb{R}^{n}$ with $|N|=0$ such that
\[f(x,p)=\lim_{\rho\rightarrow 0^{+}}\frac{m(u,l_{p},A_{\rho})}{|A_{\rho}|}\]
for all $p\in\mathbb{R}^{n}$, $x\in\mathbb{R}\setminus N$ and every family $\{A_{\rho}\}_{\rho>0}$ shrinking nicely to $x$ as $\rho\rightarrow 0^{+}$.
\end{thm}

A function $\mu:A\rightarrow\mathbb{R}$ is called subadditive if for every finite and disjoint family $(A_{i})_{i\in I}$ with $\displaystyle|A\setminus\cup_{i\in I}A_{i}|=0$,

\[\displaystyle\mu(A)\leq\sum_{i}\mu(A_{i})\]

We say that $\mu$ is dominated if $0\leq\mu(A)\leq C|A|$ for all sets $A$. Consider now the family of dominated, subadditive functions and the group of translations $(\tau_{z}\mu)(A=\mu(\tau_{z}A)$, where $\tau_{z}A=\{x\in \mathbb{R}^{n}:x-z\in A\}$.

\begin{thm}(Ergodic)(see \cite{Akcog81}, \cite{DalMod86}): Let $\mu:\mathcal{S}\rightarrow \mathbb{R}^{n}$ be a subadditive process, periodic in law, in the sense that $\mu(\cdot)$ and $\tau_{z}\mu(\cdot)$ have the same distribution for every $z\in\mathbb{Z}^{n}$. Then, there exists measurable function $\phi:\mathcal{S}\rightarrow R$ and a subset $\Omega'\subset\Omega$ of full measure such that
\[\lim_{t\rightarrow\infty}\frac{\mu(\omega)(tQ)}{\left|tQ\right|}=\phi (\omega)\] exists a.e. $\omega\in\mathcal{S}'$ and for every cube $Q\subset \mathbb{R}^{n}$. Furthermore, if $\mu$ is ergodic then $\phi$ is constant.
\end{thm}

We return to the proof of theorem $6$.

\begin{proof}
For fixed $p\in\mathbb{R}^{n}$ and for $\omega\in\mathcal{S}$ we define
\[\mu_{p}(\omega)(\Omega)=m(F(\omega),l_{p},\Omega)\]
which is a measurable map, since $m(\cdot,l_{p},\Omega)$ is continuous in $\mathcal{F}$. For $z\in\mathbb{Z}^{n}$,
\begin{align*}
(\tau_{z}\mu_{p})(\omega)(\Omega)&=m_{p}(\omega)(\tau_{z}\Omega)
\\&=\min_{u}\{(\tau_{z}F)(\omega)(\tau_{-z}u,\Omega):\tau_{-z}u-\tau_{-z}l_{p}\in W_{0,\Phi}^{1}(\Omega)\}
\\&=\min_{u}\{(\tau_{z}F)(\omega)(v+l_{p}(z),\Omega):v-l_{p}\in W_{0,\Phi}^{1}(\Omega)\}
\end{align*}
and
\[(\tau_{z}F)(\omega)(v+l_{p}(z),\Omega)=(\tau_{z}F)(\omega)(v,\Omega)\]
since the integrand is independent of $u$. Thus,
\[(\tau_{z}\mu_{p})(\omega)(\Omega)=m((\tau_{z}F),l_{p},\Omega)\]

which shows that $\mu_{p}$ is periodic in law. Applying theorem $4$, there exists measurable function $\phi:\Omega\rightarrow R$ and a subset $\Omega'\subset\Omega$ of full measure such that
\[\lim_{t\rightarrow\infty}\frac{\mu_{p}(\omega)(tQ)}{\left|tQ\right|}=\phi_{p}(\omega)\] exists a.e. $\omega\in\Omega'$ and for every cube $Q\subset \mathbb{R}^{n}$. Let
\[f_{0}(\omega,p)=\limsup_{t\rightarrow\infty}\frac{\mu_{p}(\omega)(Q_{t})}{\left|Q_{t}\right|}\]
We observe that the convexity of $F$ in $u$ says that the functions
\[p\rightarrow\frac{\mu_{p}(\omega)(\Omega)}{|\Omega|}\]
are convex and equibounded. Hence, $f_{0}$ is convex in $p$ and
\[f_{0}(\omega,p)=\lim_{t\rightarrow\infty}\frac{\mu_{p}(\omega)(tQ)}{\left|tQ\right|}\]
Furthermore,
\[\mu_{p}(\omega)(tQ)=t^{n}m((\rho_{1/t}F)(\omega),l_{p},Q)\]
so that, since $\rho_{\varepsilon}F=F_{\varepsilon}$,
\[\lim_{\varepsilon\rightarrow 0^{+}}\frac{m(F(\omega),l_{p},Q)}{\left|Q\right|}=f_{0}(\omega,p)\]
for each cube $Q$, $p\in\mathbb{R}^{n}$, $\omega\in\mathcal{S}'$. Fix $\omega\in\mathcal{S}'$. Corollary $1$ and proposition $4$ tell that there is an integral functional $F_{0}(\omega)\in\mathcal{F}$ such that $F_{\varepsilon}(\omega)$ $\Gamma-$ converges to $F_{0}(\omega)$. Then we are in position to compute the integrand of $F_{0}(\omega)$ since, from theorem $7$ there is a subset $N$ with $|N|=0$ such that
\begin{align*}
g_{0}(\omega,x,p)&=\lim_{\rho\rightarrow 0^{+}}\frac{m(F_{0}(\omega),l_{p},Q_{\rho})}{\left|Q_{\rho}\right|}
\\&=\lim_{\rho\rightarrow 0^{+}}\lim_{\varepsilon\rightarrow 0^{+}}\frac{m(F_{\varepsilon}(\omega),l_{p},Q_{\rho})}{\left|Q_{\rho}\right|}
\\&=\lim_{\rho\rightarrow 0^{+}}\lim_{\varepsilon\rightarrow 0^{+}}\frac{\mu_{p}(\omega)\left(\frac{1}{\varepsilon}Q_{\rho}\right)}{\left|\frac{1}{\varepsilon}Q_{\rho}\right|}=f_{0}(\omega,p)
\end{align*}
for all $x\in\mathbb{R}^{n}\setminus N$, $p\in\mathbb{R}^{n}$ so that
\[F_{0}(u,\Omega)=\int_{\Omega}f_{0}(\omega,Du)~dx\]
Note that if $F$ is ergodic, $\mu_{p}$ is constant.
\end{proof}
\begin{remark}
The last proof is derived in \cite{DalMod86} under the assumption that the integrand is independent of $u$. If the integrand depends on $u$, i.e. $f=f(\omega,x,u,Du)$, we can consider the function
\[f_{u}(\omega,x,Du)=f(\omega,x,u,Du)\]

and apply the last proof to see that
\[f_{0,u}(\omega,p)=\lim_{\varepsilon\rightarrow 0^{+}}\frac{m(F_{u}(\omega),l_{p},Q_{1/\varepsilon})}{|Q_{1/\varepsilon}|}\]
for all $\omega\in\mathcal{S}'$, $p\in\mathbb{R}^{n}$. Hence,
\[F_{0,u}(u,\Omega)=\int_{\Omega}f_{0}(\omega,u,Du)~dx\] and from the last limit,
\[f_{0}(\omega,u,p)=f_{0,u}(\omega,p)\]
The abstract form of this theorem can be applied to obtain homogenization results over random structures for partial differential equations defined in Orlicz-Sobolev spaces.
\end{remark}

\subsection{Examples of random functionals}

An application of theorem $3$ is the case of random two-phase domains. Such domains can be obtained for instance from the realization of Poisson processes. Construction of random domains has been studied in \cite{dkontog10}, using the connectivity function of continuum percolation theory\cite{Mees96}. In particular, we take a Poisson process $X$ with density $\lambda>0$ and we consider the realization $X(\omega)$ of the process in a given domain $\Omega\subset\mathbb{R}^{n}$, for some $\omega\in\mathcal{S}$. A connection function $g:\mathbb{R}^{+}\rightarrow [0,1]$ connects two points $x_{1}, x_{2}\in X$ with probability $g(|x_{1}-x_{2}|)$, where $|\cdot|$ denotes the Euclidean distance. Suppose that $\omega$ is a given realization for $X$ which is locally finite, i.e.
a finite number of points hits every compact set $K\subset\mathbb{R}^{n}$ almost surely:
\[P(\omega\in\Omega:\psi(K)<\infty\text{ for all compact }K\subset\mathbb{R}^{n})=1\]
where $\psi(\cdot)$ is a counting measure. Let $x_{i}\in X$ be a given point of this realization.

Consider the annulus $A=\{x\in\mathbb{R}^{n}: c_{1}\leq|x-x_{i}|\leq c_{2}\}$, where $c_{1},c_{2}$ are positive constants with $c_{1}\leq c_{2}$.

 We connect the point $x_{i}$ with all the points in $A$ that are given from $X$. For this purpose we choose the connection function
 \[g(|x-x_{i}|)=
	\begin{cases}
	 1 & \mbox{ if } c_{1}\leq|x-x_{i}|\leq c_{2} \\
   0 & \mbox{ otherwise } \\
  \end{cases}\]

 For a point $x_{j}\in A$, we denote by $l_{ij}(\omega)=l(x_{i},x_{j})$ the line segment with endpoints $x_{i},x_{j}$ and let $T_{c_{1}/2}(l_{ij})(\omega)$ the tube of radius $c_{1}/2$ surrounding $l_{ij}$. Let now $\displaystyle T(x_{i})(\omega)=\cup_{j}T_{c_{1}/2}(l_{ij})(\omega)$ and $\displaystyle G(\omega,c_{1}/2)= \cup_{i}T(x_{i})(\omega)$ for all points $x_{i}$ of the process.

 Thus, the set $G(\omega,c_{1}/2)$ is the union of random tubes obtained from the given realization of the point process. Let $\Omega(\omega,c_{1}/2)=\mathbb{R}^{n}\setminus G(\omega,c_{1}/2)$.

 We define the indicator function
\[\displaystyle a(\omega,x)=1-min\{X_{F(\omega, c_{1}/2)},1\}\] which is zero in the union of tubes and one elsewhere.

Let $\Omega$ be an open, bounded domain of $\mathbb{R}^{n}$ and consider the random functional  $\displaystyle F(\omega)(u,\Omega)=\int_{\Omega}a(\omega,x)f(\omega,Du)dx=\int_{G(\omega)\cap \Omega}f(\omega,Du)dx$ for $u\in W_{\Phi}^{1}(\Omega)$. This functional is periodic in law and independent at large distances, thus ergodic. Furthermore let $\displaystyle(\rho_{\varepsilon}F)(u,A)=\varepsilon^{n}F(\rho_{\varepsilon} u,\rho_{\varepsilon}A)$ where
 $\displaystyle(\rho_{\varepsilon}u)(x)=\frac{1}{\varepsilon}u(\varepsilon x)$, $\displaystyle(\rho_{\varepsilon}A)=\{x\in\mathbb{R}^{n}:\varepsilon x\in A\}$

Then the family
\begin{displaymath}
F^{\varepsilon}(u,\Omega)=\rho_{\varepsilon}F(u,\Omega)
\end{displaymath}
satisfies the assumptions of theorem $3$. Note that the $\rho_{\varepsilon}-$ homothetic functional is the functional obtained if we scale by $\varepsilon$ the distance between the connected points of the set $F(\omega,c_{1}/2)$ that corresponds to the union of tubes $\varepsilon F=F(\varepsilon\omega,\varepsilon c_{1}/2)$, where $\varepsilon\omega$ maps to the point measure whose support is $\{\varepsilon x_{i}\}$ and $\{x_{i}\}$ is the support of $X(\omega)$. Also, the scaling properties of this model are the same (in terms of distribution) with the model that we have if we choose

\[g_{\varepsilon}(|x-x_{i}|)=
	\begin{cases}
	1 & \mbox{ if } c_{1}\varepsilon\leq |x-x_{i}|\leq c_{2}\varepsilon \\
  0 & \mbox{otherwise} \\
  \end{cases}\]
with density function $\lambda/\varepsilon$. Let us define $G^{\varepsilon}(\omega)=\varepsilon G=G(\varepsilon\omega,\varepsilon c_{1}/2)$ and $\Omega^{\varepsilon}(\omega)=\varepsilon \Omega(\omega)=\mathbb{R}^{n}\setminus G^{\varepsilon}(\omega)$.

We may also model a domain perforated with balls at random positions but nonintersecting. Instead of constructing tubes, we let every point be the center of a ball with radius $\displaystyle\rho(\omega)\leq\min d(x_{i},x_{j})(\omega)$, where the minimum is taken over all the pairs of points $x$ of $X(\omega)$. Note that, without any affect to our proofs, we may assume that $\rho(\omega)$ is identically distributed random variable taking maximum value $\min d(x_{i},x_{j})(\omega)/4$. We consider for simplicity the first case. According to this construction, we obtain a domain randomly perforated with balls of radius and with minimal distance between them. Finally, we define

$\displaystyle G^{\varepsilon}(\omega)=\bigcup_{i\geq 1}B(\varepsilon\rho(\omega),\varepsilon x_{i})\cap \Omega$ and $\Omega^{\varepsilon}(\omega)\cap\Omega=\Omega\setminus G^{\varepsilon}(\omega)$. Note that $\text{meas}G^{\varepsilon}(\omega)$ tends to zero as $\varepsilon\rightarrow 0$.

\section{Application to homogenization problems}
\subsection{Homogenization of pde's with generalized growth conditions}
We study the homogenization of the Dirichlet problem

\begin{equation}
\begin{array}{l}
\text{div}A(x,u^{\varepsilon},Du^{\varepsilon})+B(x,u^{\varepsilon},Du^{\varepsilon})=f(x)\text{ in }\Omega^{\varepsilon}, u^{\varepsilon}-u_{0}\in W_{0,G}^{1}(\Omega^{\varepsilon})
\end{array}
\label{eq:14}
\end{equation}
where $\Omega^{\varepsilon}=\Omega^{\varepsilon}(\omega)$ is a randomly perforated domain (as in section $4.1$ for instance), $\Omega^{\varepsilon}=\Omega\setminus G^{\varepsilon}$, $G$ to be precisely defined below.
 We assume that the following structure conditions hold:

\begin{equation}
\begin{array}{l}
\displaystyle p\cdot A\geq |p|g(|p|)-a_{1}g\left(\frac{|z|}{R}\right)\frac{|z|}{R}-a_{2}
\end{array}
\label{eq:15}
\end{equation}

\begin{equation}
\begin{array}{l}
\displaystyle|A|\leq a_{3}g(|p|)+a_{4}g\left(\frac{|z|}{R}\right)+a_{5}
\end{array}
\label{eq:16}
\end{equation}

where $a_{1}$, $a_{2}$, $a_{3}$, $a_{4}$, $a_{5}$ are nonnegative constants and $g$ is a $C^{1}$ function satisfying

\begin{equation}
\begin{array}{l}
\displaystyle \delta\leq\frac{tg'(t)}{g(t)}\leq g_{0} \text{ if } t>0
\end{array}
\label{eq:17}
\end{equation}
for some $\delta>0$.

We define $\displaystyle G(t)=\int_{0}^{t}g(s)~ds$. It can be shown \cite{Lieb91} that $G$ is twice differentiable, convex function and it satisfies

\begin{gather}
\displaystyle\frac{tg(t)}{1+g_{0}}\leq G(t)\leq tg(t) \text{ if }t\geq0,\\
\displaystyle\frac{G(a)}{G(b)}\leq\frac{a}{b} \text{ if }b\geq a>0,\\
\displaystyle g(t)\leq g(2t)\leq 2^{g_{0}}g(t) \text{ if }t\geq0,\\
\displaystyle ag(b)\leq ag(a)+bg(b)\text{ if }a,b\geq0
\end{gather}

The smoothness of solutions for the equation $(5.1)$ under the above structure conditions have been derived in \cite{Lieb91}.

If in addition  $tg(t)\leq C$ for $t\geq0$, $(5.5)$ implies that $G$ satisfies the $\Delta_{2}-$condition. To see this, note that from $(5.5)$,
\[\frac{C}{t}\geq\frac{g(t)}{G(t)}=\frac{G'(t)}{G(t)}=(\log G(t))'\]
and consequently
\[\log\frac{G(2t)}{G(t)}=\int_{t}^{2t}\frac{g(s)}{G(s)}~ds\leq C\log2\]
i.e.
\[G(2t)\leq 2^{C}G(t)\]
Thus, $G\in\Delta_{2}$ with $k=2^{c}$. The additional assumption $tg(t)\leq C$ is also included in \cite{Lieb91}[Lemma $2.1$] to derive $L^{\infty}$ estimates for the solutions.

We denote by $W_{G}^{1}(\Omega)$ the Orlicz-Solobev space as defined in section $2.1$. We seek the asymptotic behavior of the family $u^{\varepsilon}$ as $\varepsilon\rightarrow 0$ in the general variational form

\begin{equation}
\begin{array}{l}
\displaystyle \inf\left\{\int_{\Omega^{\varepsilon}}f(x,u^{\varepsilon},Du^{\varepsilon})~dx:u^{\varepsilon}-u_{0}\in W_{0,G}^{1}(\Omega^{\varepsilon})\right\}
\end{array}
\label{eq:18}
\end{equation}
where $u_{0}\in C^{1}(\bar{\Omega})$, $f(x,u,p)$ is a measurable function defined in $\Omega\times\mathbb{R}\times\mathbb{R}^{n}$, continuously differentiable with respect to $u,p$ and satisfies
\begin{gather}
C_{1}G(|p|)-C_{2}G(u)\leq f(x,u,p)\leq C_{3}(1+G(u)+G(|p|)),\\
|f(x,u,p)-f(x,v,q)|
\leq C_{4}(1+g(u)+g(v)+g(|p|)+g(|q|))(|u-v|+|p-q|),\\
f(x,u,p)-f(x,u,q)-\sum_{i=1}^{n}f_{q_{i}}(x,u,q)(p_{i}-q_{i})\geq 0
\end{gather}

with $f_{p_{i}}=A_{i}$, $f_{u}=B$.

\begin{thm}[\cite{Lieb91}]
Suppose that $g$ satisfies $(5.4)$
\begin{gather}
 p\cdot A\geq |p|g(|p|)-a_{1}g(|p|)|p|-a_{2},\\
|A|\leq a_{3}g(|p|)+a_{5},\\
|B|\leq b_{0}g(|p|)|p|+b_{1}
\end{gather}

for $x\in\Omega^{\varepsilon}$, $|z|\leq M$. If $u^{\varepsilon}\in L^{\infty}\cap W_{G}^{1}$ solves $Lu^{\varepsilon}=0$ where \[Lz=\operatorname*{div}A(x,z,p)+B(x,z,p)\] with $|u^{\varepsilon}|\leq M$ in $\Omega^{\varepsilon}$. Then $u^{\varepsilon}$ is locally Holder continuous with
\[\operatorname*{osc}_{B_{\rho}}u^{\varepsilon}\leq C\left(\frac{\rho}{R}\right)^{\alpha}\left(\operatorname*{osc}_{B_{R}}u^{\varepsilon}+\chi R\right)\]
for $\alpha>0$, $C=C(a_{3},b_{0}M,g_{0},n,\delta)$ and concentric balls $B_{\rho}$, $B_{R}$ in $\Omega^{\varepsilon}$, $0<\rho\leq R\leq 1$.
\end{thm}

\subsection{Passing the limit}
In this section, we improve the homogenization method of \cite{Khrus05}. For the convinience of the reader we present it explicitely.
Taking $u_{0}\in C^{1}(\Omega)$ and using the smoothness of the solution from the last section, we can extend $u^{\varepsilon}$ to $\Omega$ by taking $u^{\varepsilon}=u_{0}$ in $\Omega\setminus\Omega^{\varepsilon}$. We keep the notation $u^{\varepsilon}$ for the extended function. Then,
\[\|u^{\varepsilon}\|_{W_{G}^{1}(\Omega)}\leq c\|f\|_{W_{G}^{1}(\Omega)}\]
and hence, the family $u^{\varepsilon}$ is compact in $C^{\alpha}(\Omega)$ and weakly compact in $W_{G}^{1}(\Omega)$. Thus, we can extract a subsequence, still denoted by $u^{\varepsilon}$ that converges weakly to a function $u\in C^{\alpha}(\Omega)\cap W_{G}^{1}(\Omega)$.

We denote by $Q_{h}^{x}=Q(x,h)$ the cube centered at $x$ of size $h>0$ and we define the function
\begin{equation}
\begin{array}{l}
\displaystyle R(x,v^{\varepsilon},\nabla v^{\varepsilon})=f(x,0,\nabla v^{\varepsilon})+h^{-1-\gamma}G(v^{\varepsilon}-b)
\end{array}
\label{eq:24}
\end{equation}

where $\gamma$ is a positive parameter. We also define the capacity-type functionals
\begin{equation}
\begin{array}{l}
\displaystyle \text{cap}_{f}(\omega,x,\varepsilon,h,b)=\inf_{v^{\varepsilon}}\int_{Q_{x}^{h}\cap\Omega^{\varepsilon}} f(x,0,\nabla v^{\varepsilon})~dx
\end{array}
\label{eq:25}
\end{equation}
and
\begin{equation}
\begin{array}{l}
\displaystyle \operatorname*{cap}(\omega,x,\varepsilon,h,b)=\inf_{v^{\varepsilon}}\int_{Q_{x}^{h}\cap\Omega^{\varepsilon}} R(x,v^{\varepsilon},\nabla v^{\varepsilon})~dx
\end{array}
\label{eq:26}
\end{equation}
where the infimum is taken over the set $\{v^{\varepsilon}\in W_{G}^{1}(\Omega):v^{\varepsilon}=0\text{ in }\Omega\setminus\Omega^{\varepsilon}\}$. Note that theorem $3$ shows that the limit
\[\lim_{h\rightarrow 0}\lim_{\varepsilon\rightarrow 0}\frac{\operatorname*{cap}_{f}(\omega,x,\varepsilon,h,b)}{h^{n}}=c(x)\]
exists for every point $x\in\Omega$. In addition, since our functionals are independent at large distances, the ergodicity implies that the limit $c(x)=c_{0}$ is constant. The continuity of $v^{\varepsilon}$ shows that the integral $\displaystyle\int_{Q_{x}^{h}\cap\Omega^{\varepsilon}}G(v^{\varepsilon}-b)~dx$ can estimated in terms of $\displaystyle\int_{Q_{x}^{h}\cap\Omega^{\varepsilon}}G(|\nabla v^{\varepsilon}|)~dx$ (see also \cite{Xing10} for Poincare type inequalities) times a factor $h^{1+\gamma}$.

Let us also consider the limit
\[\lim_{h\rightarrow 0}\lim_{\varepsilon\rightarrow 0}\frac{\operatorname*{cap}(\omega,x,\varepsilon,h,b)}{h^{n}}=c_{0}(b)\]

Furthermore, we assume that for every $x\in\Omega$,
\[\limsup_{\varepsilon\rightarrow 0}\frac{\operatorname*{cap}(\omega,x,\varepsilon,h,b)}{h^{n}}\leq C(1+g(|b|))|b|\]

\begin{thm}
The (extended) family of minimizers of $(5.9)$ converges weakly to the minimizer $u\in C^{\alpha}(\Omega)\cap W_{G}^{1}(\Omega)$ of
\begin{equation}
\begin{array}{l}
\displaystyle \inf\left\{\int_{\Omega}f(x,u,Du)+c_{0}(u-u_{0})~dx:u-u_{0}\in W_{0,G}^{1}(\Omega)\right\}
\end{array}
\label{eq:27}
\end{equation}
\end{thm}

\begin{proof}
We consider a partition of $\Omega$ with cubes $Q^{\alpha}=Q(x^{\alpha},h)$ centered at $x^{\alpha}$ of size $h$, so that $\displaystyle\cup_{\alpha}Q(x^{\alpha},h)$ is a cover of $\Omega$ and the points $x^{\alpha}$ form a periodic lattice of period $h-r$, $r$ to be chosen. Consider a partition of unity $\{\phi_{\alpha}\}$ of $C^{2}$ functions such that
\begin{enumerate}
\item $0\leq\phi_{\alpha}\leq 1$
\item $\phi_{\alpha}=0$ if $x\notin Q^{\alpha}$, $\phi_{\alpha}=1$ if $\displaystyle x\in Q^{\alpha}\setminus\cup_{\beta\neq\alpha}Q^{\beta}$
\item $\displaystyle\sum_{\alpha}\phi_{\alpha}(x)=1$, if $x\in D$
\item $\left|\nabla\phi_{\alpha}\right|\leq C/r$
\end{enumerate}
Let us denote by $v_{\alpha}=v_{\alpha}^{\varepsilon}$ the minimizer of $(6.9)$ in the cube centered at $x^{\alpha}$ with $b=b_{\alpha}$.

From the last assumption and $(6.10)$,
\[\int_{Q^{\alpha}}G(|\nabla v_{\alpha}|)~dx=O(h^{n})\]
and
\[\int_{Q^{\alpha}}G(|v_{\alpha}-b_{\alpha}|)~dx=O(h^{n+1+\gamma})\]

We denote by $\displaystyle\hat{Q}_{h}^{\alpha}=Q_{h}^{\alpha}\setminus\cup_{\beta\neq\alpha}Q_{h}^{\beta}$  the concentric cube centered at $x^{\alpha}$ of size $\hat{h}=h-2r$.
Then,
\begin{align*}&\int_{Q_{h}^{\alpha}\setminus\hat{Q}_{h}^{\alpha}}R(x,v_{\alpha},\nabla v_{\alpha})~dx
\\&=\int_{Q_{h}^{\alpha}}R(x,v_{\alpha},\nabla v_{\alpha})~dx-\int_{\hat{Q}_{h}^{\alpha}}R(x,v_{\alpha},\nabla v_{\alpha})~dx+O(rh^{n-1})
\\&\leq\text{cap}(\omega,x,h,\varepsilon,b)-\text{cap}(\omega,x,\hat{h},\varepsilon,b)+O(rh^{n-1})\end{align*}
So,
\[\int_{(Q_{h}^{\alpha}\setminus\hat{Q}_{h}^{\alpha})\cap\Omega^{\varepsilon}}R(x,v_{\alpha},\nabla v_{\alpha})~dx=o(h^{n})\]
which implies that
\begin{gather}
\int_{(Q_{h}^{\alpha}\setminus\hat{Q}_{h}^{\alpha})\cap\Omega^{\varepsilon}}G(|\nabla v_{\alpha}|)~dx=o(h^{n})\\
\int_{(Q_{h}^{\alpha}\setminus\hat{Q}_{h}^{\alpha})\cap\Omega^{\varepsilon}}G(v_{\alpha}-b_{\alpha})~dx=o(h^{n})h^{1+\gamma}
\end{gather}
Note that these relations imply that the local limit
\[\lim_{h\rightarrow 0}\lim_{\varepsilon\rightarrow 0}\frac{\operatorname*{cap}(\omega,x,\varepsilon,h,b)}{h^{n}}=c(x,b)=c_{0}(b)\]
exists.

Consider a function $w\in C^{1}(\Omega)$ such that $w=u_{0}$ on $\partial\Omega$ and denote by $\mathcal{K}_{\theta}$ the set of the cubes $Q_{h}^{\alpha}$ that cover $\Omega$ such that
$|w(x)-f(x)|>\theta$ for $\theta>0$. We set $b_{\alpha}=w(x^{\alpha})-f(x^{\alpha})$ for each $Q_{h}^{\alpha}\in\mathcal{K}_{\theta}$ and $b_{\alpha}=1$ for $Q_{h}^{\alpha}\notin\mathcal{K}_{\theta}$. For each cube $Q_{h}^{\alpha}$, we define the set
\begin{equation}
\begin{array}{l}
\displaystyle\mathcal{B}^{\alpha}(\delta,\varepsilon,h)=\{x\in Q_{h}^{\alpha}:\operatorname*{sgn}b_{\alpha}\leq|b_{\alpha}|-\delta\}
\end{array}
\label{eq:28}
\end{equation}
 and the function
 \[V_{\alpha}^{\varepsilon}(x)=
 \begin{cases}
 v_{\alpha}^{\varepsilon} & \mbox{ if }x\in\mathcal{B}^{\alpha}(\delta,\varepsilon,h)\\
 b_{\alpha}^{\delta}=(|b_{\alpha}|-\delta)\operatorname*{sgn}b_{\alpha} & \mbox{ if }x\in Q_{h}^{\alpha}\setminus\mathcal{B}^{\alpha}(\delta,\varepsilon,h)
 \end{cases}\]
with $0<\delta\leq\theta/2\ll 1$. To estimate the measure of $\mathcal{B}^{\alpha}(\delta,\varepsilon,h)$, observe that from our assumption and $(5.26)$, for $\varepsilon$ sufficiently small,
\begin{align*}
G(\delta)\operatorname*{meas}\mathcal{B}^{\alpha}(\delta,\varepsilon,h)\leq\int_{\mathcal{B}^{\alpha}(\delta,\varepsilon,h)\cap\Omega^{\varepsilon}}&G(v_{\alpha}-b_{\alpha})~dx
\\&\leq\int_{Q_{h}^{\alpha}\cap\Omega^{\varepsilon}}G(v_{\alpha}-b_{\alpha})~dx\leq Ch^{n+1+\gamma}
\end{align*}
We set $\delta=G^{-1}((h^{-1-\gamma})^{-1+\delta_{1}})$ for $\delta_{1}\in(0,1)$. Then

\begin{equation}
\begin{array}{l}
\displaystyle\operatorname*{meas}\mathcal{B}^{\alpha}(\delta,\varepsilon,h)=O(h^{n})h^{(1+\gamma)\delta_{1}}=o(h^{n})
\end{array}
\label{eq:29}
\end{equation}

as $h\rightarrow 0$. For $w\in C^{2}(\Omega)$ supported in $\Omega$, consider the function
\[w_{h}^{\varepsilon}(x)=w(x)+\sum_{\alpha}\frac{(w(x)-u_{0}(x))}{b_{\alpha}^{\delta}}(V_{\alpha}^{\varepsilon}(x)-b_{\alpha}^{\delta})\phi_{\alpha}(x)\]
so that $w_{h}^{\varepsilon}-u_{0}\in W_{0,G}^{1}(\Omega^{\varepsilon})$. Let us denote by
\[J^{\varepsilon}[v^{\varepsilon}]=\int_{\Omega^{\varepsilon}}f(x,v^{\varepsilon},Dv^{\varepsilon})~dx\]
and note that since $u^{\varepsilon}$ minimizes $J^{\varepsilon}[\cdot]$,
\[J^{\varepsilon}[u^{\varepsilon}]\leq J^{\varepsilon}[w_{h}^{\varepsilon}]\]

To estimate the last inequality, observe that
\[J^{\varepsilon}[w_{h}^{\varepsilon}]\leq\sum_{\alpha}\int_{\hat{Q}_{h}^{\alpha}\cap\Omega^{\varepsilon}}f(x,w_{h}^{\varepsilon},Dw_{h}^{\varepsilon})~dx+\sum_{\alpha,\beta}\int_{(\hat{Q}_{h}^{\alpha}\cap\hat{Q}_{h}^{\beta})\cap\Omega^{\varepsilon}}|f(x,w_{h}^{\varepsilon},Dw_{h}^{\varepsilon})|~dx\]

Taking into account the properties of $V_{\alpha}^{\varepsilon},\phi_{\alpha},w$ and the convexity of $G(\cdot)$ one derives
\begin{equation}
\begin{array}{l}
\displaystyle\lim_{h\rightarrow 0}\limsup_{\varepsilon\rightarrow 0}\sum_{\alpha,\beta}\int_{(\hat{Q}_{h}^{\alpha}\cap\hat{Q}_{h}^{\beta})\cap\Omega^{\varepsilon}}|f(x,w_{h}^{\varepsilon},Dw_{h}^{\varepsilon})|~dx=0
\end{array}
\label{eq:30}
\end{equation}

Introduce the sets $\mathcal{B}_{1}^{\alpha}(\varepsilon,h)=\mathcal{B}^{\alpha}(\delta,\varepsilon,h)\cap(\hat{Q}_{h}^{\alpha}\cap\Omega^{\varepsilon})$ and $\mathcal{B}_{2}^{\alpha}(\varepsilon,h)=(\hat{Q}_{h}^{\alpha}\cap\Omega^{\varepsilon})\setminus \mathcal{B}_{1}^{\alpha}(\varepsilon,h)$. Then,

\begin{equation}
\begin{array}{l}
\displaystyle\int_{\mathcal{B}_{2}^{\alpha}(\varepsilon,h)}f(x,w_{h}^{\varepsilon},Dw_{h}^{\varepsilon})~dx\leq\int_{\hat{Q}_{h}^{\alpha}}f(x,w,Dw)~dx+o(h^{n})
\end{array}
\label{eq:31}
\end{equation}

Furthermore, for $Q_{h}^{\alpha}\in\mathcal{K}_{\theta}$, we have
\begin{align*}
\displaystyle\int_{\mathcal{B}_{1}^{\alpha}(\varepsilon,h)}&f(x,w_{h}^{\varepsilon},Dw_{h}^{\varepsilon})~dx
\\&=\int_{\mathcal{B}_{1}^{\alpha}(\varepsilon,h)}f(x,0,Dv_{\alpha}^{\varepsilon})~dx
+\int_{\mathcal{B}_{1}^{\alpha}(\varepsilon,h)}f(x,w_{h}^{\varepsilon},Dw_{h}^{\varepsilon})-f(x,0,Dv_{\alpha}^{\varepsilon})~dx
\\&\leq\int_{\mathcal{B}_{1}^{\alpha}(\varepsilon,h)}f(x,0,Dv_{\alpha}^{\varepsilon})~dx+c_{1}\int_{\mathcal{B}_{1}^{\alpha}(\varepsilon,h)}g(c_{2}+c_{3}|\nabla v_{\alpha}^{\varepsilon}|)~dx
\\&+O(h)\int_{\mathcal{B}_{1}^{\alpha}(\varepsilon,h)}G(|\nabla v_{\alpha}^{\varepsilon}|)~dx
\end{align*}
 for constants independent of $\varepsilon,h,\delta$. Note that the second term on the right hand side is of order $O(h^{n+1})$. To estimate the first term consider the subsets of $\mathcal{B}_{1}^{\alpha}(\varepsilon,h)$,
\begin{gather}
\displaystyle\mathcal{B}_{11}^{\alpha}(\varepsilon,h)=\{x\in\mathcal{B}_{1}^{\alpha}(\varepsilon,h):c_{2}+c_{3}|\nabla v_{\alpha}^{\varepsilon}|\leq m(h)\}\\
\displaystyle\mathcal{B}_{12}^{\alpha}(\varepsilon,h)=\{x\in\mathcal{B}_{1}^{\alpha}(\varepsilon,h):c_{2}+c_{3}|\nabla v_{\alpha}^{\varepsilon}|> m(h)\}
\end{gather}
where $m(h)=(h^{-1-\gamma})^{\delta_{1}/2}$. From condition $5$ in section $2$,
\begin{align*}
m(h)g(c_{2}+c_{3}|\nabla v_{\alpha}^{\varepsilon}|)&<(c_{2}+c_{3}|\nabla v_{\alpha}^{\varepsilon}|)g(c_{2}+c_{3}|\nabla v_{\alpha}^{\varepsilon}|)
\\&\leq CG(c_{2}+c_{3}|\nabla v_{\alpha}^{\varepsilon}|)\end{align*}

in $\mathcal{B}_{12}^{\alpha}(\varepsilon,h)$. The convexity of $G$ and $(6.27)$ give us

\begin{equation}
\begin{array}{l}
\displaystyle\int_{\mathcal{B}_{12}^{\alpha}(\varepsilon,h)}g(c_{2}+c_{3}|\nabla v_{\alpha}^{\varepsilon}|)~dx=o(h^{n})
\end{array}
\label{eq:32}
\end{equation}
so that
\begin{equation}
\begin{array}{l}
\displaystyle\int_{Q_{h}^{\alpha}\cap\Omega^{\varepsilon}}f(x,w_{h}^{\varepsilon},Dw_{h}^{\varepsilon})~dx\leq\int_{Q_{h}^{\alpha}\cap\Omega}f(x,w_{h}^{\varepsilon},Dw_{h}^{\varepsilon})~dx+\operatorname*{cap}(\omega,x^{\alpha},\varepsilon,h,b_{\alpha})+o(h^{n})
\end{array}
\label{eq:33}
\end{equation}
for cubes $Q_{h}^{\alpha}\in\mathcal{K}_{\theta}$. Similarly we can obtain this inequality for $Q_{h}^{\alpha}\notin\mathcal{K}_{\theta}$.
Letting $\varepsilon\rightarrow 0$, $h\rightarrow 0$ and finally $\theta\rightarrow 0$, we obtain

\begin{equation}
\begin{array}{l}
\displaystyle\lim_{h\rightarrow 0}\limsup_{\varepsilon\rightarrow 0}J^{\varepsilon}[w_{h}^{\varepsilon}]\leq J_{c_{0}}(w)
\end{array}
\label{eq:34}
\end{equation}
and hence
\begin{equation}
\begin{array}{l}
\displaystyle\lim_{\varepsilon\rightarrow 0}J^{\varepsilon}[u^{\varepsilon}]\leq J_{c_{0}}(w)
\end{array}
\label{eq:35}
\end{equation}
for all $w\in C^{1}(\Omega)$ and hence for all $w\in W_{G}^{1}(\Omega)$, with $J_{c_{0}}(u)$ being the functional in $(5.23)$, which is continuous in $W_{G}^{1}(\Omega)$.

To prove the lower bound, we will use the following lemma \cite{Khrus05}:
\begin{lemma}
Suppose $w\in W_{G}^{1}(\Omega)$ with $\displaystyle\|w\|_{G,\Omega}^{1}< 1$. Under the assumption of theorem $8$, there exists a sequence $W^{\varepsilon}\in W_{G}^{1}(\Omega)$ with $W^{\varepsilon}=0$ in $\Omega\setminus\Omega^{\varepsilon}$ converging weakly in $W_{G}^{1}(\Omega)$ to $w$ and, for sufficiently small $\varepsilon$, satisfies
\[\|W^{\varepsilon}\|_{G,\Omega}^{1}\leq\Lambda\left(\|w\|_{G,\Omega}^{1}\right)\]
for some continuous nonnegative function $\Lambda$ with $\displaystyle\lim_{t\rightarrow 0}\Lambda(t)=0$.
\end{lemma}

Now let $u\in W_{G}^{1}(\Omega)$ be a weak limit of a sequence $u^{\varepsilon}$ of solutions of $(5.9)$ extended by $u_{0}$ in $\Omega\setminus\Omega^{\varepsilon}$. Given $\delta>0$ we pick a function $u_{\delta}\in C^{1}(\Omega)$ such that
\begin{equation}
\begin{array}{l}
\displaystyle\|u_{\delta}-u\|_{G,\Omega}^{1}<\delta
\end{array}
\label{eq:36}
\end{equation}
From lemma $4$, it follows that there exists a sequence $\{W_{\delta}^{\varepsilon}\}$ converging weakly to $u_{\delta}-u$. We define $u_{\delta}^{\varepsilon}=u^{\varepsilon}+W_{\delta}^{\varepsilon}$ so that, as $\varepsilon\rightarrow 0$, $u_{\delta}^{\varepsilon}\rightharpoonup u_{\delta}$ and $u_{\delta}^{\varepsilon}=u_{0}$ in $\Omega\setminus\Omega^{\varepsilon}$.
Then,
\[\lim_{\delta\rightarrow 0}\limsup_{\varepsilon\rightarrow 0}\|u_{\delta}^{\varepsilon}-u^{\varepsilon}\|_{G,\Omega}^{1}=0\] that leads to
\[\lim_{\delta\rightarrow 0}\limsup_{\varepsilon\rightarrow 0}|J^{\varepsilon}[u_{\delta}^{\varepsilon}]-J^{\varepsilon}[u^{\varepsilon}]|=0\]
The continuity of $J_{c_{0}}$ implies that it suffices to show
\begin{equation}
\begin{array}{l}
\displaystyle\lim_{\varepsilon\rightarrow 0}J_{c_{0}}[u^{\varepsilon}]\geq J_{c_{0}}[u]
\end{array}
\label{eq:37}
\end{equation}
To prove $(6.37)$, we introduce the sets
\[\Omega_{\theta}=\Omega_{\theta}^{+}\cup\Omega_{\theta}^{-} \qquad\qquad \tilde{\Omega}_{\theta}=\tilde{\Omega}_{\theta}^{+}\cup\tilde{\Omega}_{\theta}^{-}\qquad\qquad \Xi_{\theta}=\Omega\setminus\Omega_{\theta}\]
\[\Omega_{\theta}^{\varepsilon}=\Omega_{\theta}\cup\Omega^{\varepsilon} \qquad\qquad \tilde{\Omega}_{\theta}^{\varepsilon}=\tilde{\Omega}_{\theta}\cup\tilde{\Omega}^{\varepsilon}\qquad\qquad \Xi_{\theta}^{\varepsilon}=\Xi_{\theta}\cap\Omega_{\varepsilon}\]
where
\[\Omega_{\theta}^{\pm}=\left\{c\in\Omega:\pm(u_{\delta}-u_{0})>\theta\right\}\qquad\qquad \tilde{\Omega}_{\theta}^{\pm}=\left\{\cup_{\alpha} Q_{h}^{\alpha}:Q_{h}^{\alpha}\subset\Omega_{\theta}^{\pm}\right\}\]
Since $u_{\delta}$ is smooth,
\[\lim_{h\rightarrow 0}\operatorname*{meas}[\Omega_{\theta}^{\varepsilon}\setminus\tilde{\Omega}_{\theta}^{\varepsilon}]=0\]
The limit of the capacity functional implies that
\[\operatorname*{meas}[G^{\varepsilon}\cap Q_{h}^{\alpha}]=O(h^{n})h^{1+\gamma}\] for sufficiently small $\varepsilon>0$ so that
\[\operatorname*{meas}[G^{\varepsilon}]=o(1)\]
We write $J^{\varepsilon}[u_{\delta}^{\varepsilon}]$ in the following way:
\begin{align*}J^{\varepsilon}[u_{\delta}^{\varepsilon}]=\int_{\tilde{\Omega}_{\theta}^{\varepsilon}}&f(x,u_{\delta}^{\varepsilon},\nabla u_{\delta}^{\varepsilon})~dx
\\&+\int_{\Omega^{\varepsilon}\setminus\tilde{\Omega}_{\theta}^{\varepsilon}}f(x,u_{\delta}^{\varepsilon},\nabla u_{\delta}^{\varepsilon})~dx+\int_{\Xi_{\theta}^{\varepsilon}}f(x,u_{\delta}^{\varepsilon},\nabla u_{\delta}^{\varepsilon})~dx
\end{align*}
From $(6.11)$ and $(6.12)$ we have
\begin{align*}
f(x,u_{\delta}^{\varepsilon},\nabla u_{\delta}^{\varepsilon})\geq f(&x,u_{\delta},\nabla u_{\delta})  \\&+\sum_{i=1}^{n}\frac{\partial f}{\partial u_{x_{i}}}(x,u_{\delta},\nabla u_{\delta})\left(\frac{\partial u_{\delta}^{\varepsilon}}{\partial x_{i}}-\frac{\partial u_{\delta}}{\partial x_{i}}\right)
\\&-(g(u_{\delta}^{\varepsilon})+2g(|\nabla u_{\delta}^{\varepsilon}|)+1)(|u_{\delta}^{\varepsilon}-u_{\delta}|)
\end{align*}
The convergence of $u_{\delta}^{\varepsilon}$ to $u_{\delta}$ and the last estimation implies that
\[\lim_{h\rightarrow 0}\liminf_{\varepsilon\rightarrow 0}\int_{\Omega^{\varepsilon}\setminus\tilde{\Omega}_{\theta}^{\varepsilon}}f(x,u_{\delta}^{\varepsilon},\nabla u_{\delta}^{\varepsilon})~dx\geq 0\]

Furthermore,
\[\int_{\Xi_{\theta}^{\varepsilon}}f(x,u_{\delta}^{\varepsilon},\nabla u_{\delta}^{\varepsilon})~dx=\int_{\Xi_{\theta}}f(x,u_{\delta}^{\varepsilon},\nabla u_{\delta}^{\varepsilon})~dx-\int_{\Xi_{\theta}\cap G^{\varepsilon}}f(x,u_{\delta}^{\varepsilon},\nabla u_{\delta}^{\varepsilon})~dx\]

and

\[\liminf_{\varepsilon\rightarrow 0}\int_{\Xi_{\theta}^{\varepsilon}}f(x,u_{\delta}^{\varepsilon},\nabla u_{\delta}^{\varepsilon})~dx\geq\int_{\Xi_{\theta}}f(x,u_{\delta},\nabla u_{\delta})~dx\]

Let $Q_{h}^{\alpha}$ be a cube in $\tilde{\Omega}_{\theta}^{+}$ and set
\[b_{\alpha}^{\min}=\min_{Q_{h}^{\alpha}}u_{\delta}-h_{1}\qquad b_{\alpha}^{\max}=\max_{Q_{h}^{\alpha}}u_{0}+h_{1}\qquad b_{\alpha}=b_{\alpha}^{\min}-b_{\alpha}^{\max}\]
for $h_{1}>0$ to be chosen. Now we split the set $Q_{h}^{\alpha}\cap\Omega^{\varepsilon}$ into the following nonintersecting sets:
\[\Omega_{1\alpha}^{\varepsilon}=\left\{x\in Q_{h}^{\alpha}\cap\Omega^{\varepsilon}:u_{\delta}^{\varepsilon}<b_{\alpha}^{\max}\right\}\]
\[\Omega_{2\alpha}^{\varepsilon}=\left\{x\in Q_{h}^{\alpha}\cap\Omega^{\varepsilon}:b_{\alpha}^{\max}\leq u_{\delta}^{\varepsilon}\leq b_{\alpha}^{\min}\right\}\]
\[\Omega_{3\alpha}^{\varepsilon}=\left\{x\in Q_{h}^{\alpha}\cap\Omega^{\varepsilon}:u_{\delta}^{\varepsilon}> b_{\alpha}^{\min}\right\}\]

Since $u_{\delta}^{\varepsilon}\rightarrow u_{\delta}$ in the $L_{G}(\Omega)$, for $\varepsilon>0$ small enough
\[\int_{Q_{h}^{\alpha}}G(u_{\delta}^{\varepsilon}-u_{\delta})=O(h^{n+2+2\gamma})\] and hence
\[G(h_{1})\operatorname*{meas}[\Omega_{1\alpha}^{\varepsilon}\cup\Omega_{2\alpha}^{\varepsilon}]\leq\int_{\Omega_{1\alpha}^{\varepsilon}\cup\Omega_{2\alpha}^{\varepsilon}}G(u_{\delta}^{\varepsilon}-u_{\delta})=O(h^{n+2+2\gamma})\]
Choosing $h_{1}=G^{-1}(h^{1+\gamma})$, we get $\operatorname*{meas}[\Omega_{1\alpha}^{\varepsilon}\cup\Omega_{2\alpha}^{\varepsilon}]=O(h^{n+1+\gamma})$
We follow the steps that we did to show $(6.33)$, to get
\begin{equation}
\begin{array}{l}
\displaystyle\liminf_{\varepsilon\rightarrow 0}\int_{\Omega_{1\alpha}^{\varepsilon}\cup\Omega_{3\alpha}^{\varepsilon}}f(x,u_{\delta}^{\varepsilon},\nabla u_{\delta}^{\varepsilon})~dx\geq\int_{Q_{h}^{\alpha}}f(x,u_{\delta},\nabla u_{\delta})~dx+o(h^{n})
\end{array}
\label{eq:38}
\end{equation}

The estimate over $\Omega_{2\alpha}^{\varepsilon}$ follows by introducing the function
\[v_{\alpha}^{\varepsilon}=
\begin{cases}
 0, & \mbox{ in }\Omega_{3\alpha}^{\varepsilon}\cup(G^{\varepsilon}\cup Q_{h}^{\alpha})\\
 u_{\delta}^{\varepsilon}-b_{\alpha}^{\max}, & \mbox{ in } \Omega_{2\alpha}^{\varepsilon}\\
 b_{\alpha}, & \mbox{ in }\Omega_{3\alpha}^{\varepsilon}
\end{cases}\]

Since $u_{\delta}^{\varepsilon}$ are bounded in $\Omega_{2\alpha}^{\varepsilon}$,
\begin{align*}\int_{\Omega_{2\alpha}^{\varepsilon}}f(x,u_{\delta}^{\varepsilon},\nabla u_{\delta}^{\varepsilon})~dx=&\int_{Q_{h}^{\alpha}}f(x,0,\nabla v_{\alpha}^{\varepsilon})~dx  \\&+h^{-1-\gamma}\int_{Q_{h}^{\alpha}}G(v_{\alpha}^{\varepsilon}-b_{\alpha})~dx+o(h^{n})
\end{align*}
so that
\[\int_{\Omega_{2\alpha}^{\varepsilon}}f(x,u_{\delta}^{\varepsilon},\nabla u_{\delta}^{\varepsilon})~dx\geq\operatorname*{cap}(\omega,x^{\alpha},\varepsilon,h,b_{\alpha})+o(h^{n})\]

It follows that
\begin{align*}\liminf_{\varepsilon\rightarrow 0}\int_{Q_{h}^{\alpha}\cap\Omega^{\varepsilon}}f(x,u_{\delta}^{\varepsilon},\nabla u_{\delta}^{\varepsilon})~dx\geq\int_{Q_{h}^{\alpha}}&f(x,u_{\delta},\nabla u_{\delta})~dx
\\&+\liminf_{\varepsilon\rightarrow 0}\operatorname*{cap}(\omega,x^{\alpha},\varepsilon,h,b_{\alpha})+o(h^{n})\end{align*}

The same inequality can be obtained for $Q_{h}^{\alpha}\subset\Omega_{\theta}^{-}$ so that
\[J_{c_{0}}[u]\leq J_{c_{0}}[w]\] for all $w\in W_{0,G}^{1}(\Omega)$. This completes the proof.
\end{proof}

\subsection{p(x)-Laplace type equations}
Homogenization results for the $p(x)-$Laplacian have been studied in \cite{MR2569903}. The problem under consideration is
\begin{gather}
-\operatorname*{div}(|\nabla u^{\varepsilon}|^{p_{\varepsilon}(x)-2}\nabla u^{\varepsilon})+|u^{\varepsilon}|^{p_{\varepsilon}(x)-2}u^{\varepsilon}=f(x) \mbox{ in }\Omega^{\varepsilon}\\
u^{\varepsilon}=0 \mbox{ on }\partial\Omega^{\varepsilon}
\end{gather}
where $p_{\varepsilon}(x)$ is continuous, oscillating function satisfying a modulus of continuity $|p_{\varepsilon}(x)-p_{\varepsilon}(y)|\leq\xi_{\varepsilon}(|x-y|)$ with $\displaystyle\limsup_{\tau\rightarrow 0}\xi_{\varepsilon}(\tau)\ln(1/\tau)=0$. Such equations are defined in Sobolev spaces with variable exponents. Assuming certain convergence properties on $p_{\varepsilon}(\cdot)$, it is shown that the sequence $u_{\varepsilon}$ converges in $W^{1,p_{0}(x)}(\Omega)$, $\displaystyle\lim_{\varepsilon\rightarrow 0}\|p_{\varepsilon}-p_{0}\|_{C^{0}(\Omega)}=0$, to the minimizer $u$ of the energy functional
\[\int_{\Omega}\frac{1}{p_{0}(x)}|\nabla u|^{p_{0}(x)}+\frac{1}{p_{0}(x)}|u|^{p_{0}(x)}+c(x,u)-fu~dx\]
where
\[c(x,b)=\lim_{h\rightarrow 0}\lim_{\varepsilon\rightarrow 0}\frac{c(\varepsilon,h,z,b)}{h^{n}}\]
and
\[c(\varepsilon,h,z,b)=\inf_{v^{\varepsilon}}\int_{Q_{h}^{z}}\frac{1}{p_{\varepsilon}(x)}|\nabla u|^{p_{\varepsilon}(x)}+h^{-1-\gamma}(|v^{\varepsilon}-b|^{p_{\varepsilon}(x)}+|v^{\varepsilon}-b|^{p_{0}(x)})~dx\]

Proceeding as in the proof of theorem $9$, one can show the existence of the capacity limits, assuming appropriately modelled perforated domains in which the ergodic theorem is valid. A setup of $\Gamma-$convergence in Sobolev spaces with variable exponent and the integral representation of variational functionals (see \cite{Coscia02}) would produce similar homogenization results as long as theorem $4$ is proved.

\bibliographystyle{plain}

\bibliography{hcp}
%\specialchapt{ACKNOWLEDGEMENTS}

\end{document}